\PassOptionsToPackage{pagebackref}{hyperref}
\documentclass[%
enabledeprecatedfontcommands,
final,
]{mpi2015-cscpreprint}

\usepackage[american]{babel}

\usepackage{graphicx}

\usepackage{amssymb}
\usepackage{amsthm}
\usepackage{orcidlink}
\usepackage{manuscript}
\nolinenumbers

\newtheorem{theorem}{Theorem}[section]
\newtheorem{corollary}[theorem]{Corollary}
\newtheorem{example}[theorem]{Example}
\newtheorem{lemma}[theorem]{Lemma}
\newtheorem{proposition}[theorem]{Proposition}
\newtheorem{remark}[theorem]{Remark}

\hypersetup{bookmarksnumbered}

\begin{document}


\title{A unifying framework for ADI-like methods for linear matrix equations and beneficial consequences}

\author[$\dagger\ast$]{Jonas Schulze \orcidlink{0000-0002-2086-7686}}
\author[$\dagger$]{Jens Saak \orcidlink{0000-0001-5567-9637}}
\affil[$\dagger$]{Max Planck Institute for Dynamics of Complex Technical Systems, \authorcr%
  Sandtorstr. 1, 39106 Magdeburg, Germany.}
\affil[$\ast$]{Corresponding author.\ \email{jschulze@mpi-magdeburg.mpg.de}}

\shortauthor{J.~Schulze, J.~Saak}
\shortdate{}



\abstract{
%

We derive the \ac{ADI} method based on a commuting operator split
and apply the results \chadded{in detail} to the continuous time algebraic \Lyapunov{} equation
with low-rank constant term and approximate solution,
\chadded{giving pointers for the Sylvester case.}
Previously, it has been mandatory to start the low-rank ADI
\chreplaced{%
  for Lyapunov equations (CF-ADI, LR-ADI, G-LR-ADI)\@ or Sylvester equations (fADI, G-fADI)%
}{%
  (LR-ADI)%
}
with an all-zero initial value.
Our approach \chreplaced{extends}{retains} the known efficient iteration schemes of low-rank increments and residuals
to arbitrary low-rank initial values for \chreplaced{all these methods}{the LR-ADI method}.
We further generalize
\chreplaced{%
  two properties of the low-rank \Lyapunov{} ADI
  to the generic ADI applied to arbitrary linear equations using a commuting operator split, namely
  the invariance of iterates under permutations of the shift parameters, and
  the efficient handling of complex shift parameters.%
}{%
  some of the known properties of the LR-ADI for \Lyapunov{} equations
  to larger classes of algorithms or problems.%
}

We investigate the performance of arbitrary initial values using two outer
iterations in which \chreplaced{the low-rank \Lyapunov{} ADI}{LR-ADI} is typically called.
First, we solve an algebraic \Riccati{} equation with the Newton method.
Second, we solve a differential \Riccati{} equation with a first-order Rosenbrock method.
Numerical experiments confirm that the proposed new initial value of the \ac{ADI}
can lead to a significant reduction in the total number of \ac{ADI} steps,
while also showing a \qty{17}{\%} and $8\times$ speed-up over the zero initial value for the two equation types, respectively.
}

\novelty{%
We introduce the notion of fully commuting splitting schemes to solve arbitrary linear systems,
and derive the \ac{ADI} method in that context.
This allows us to extend the low-rank \Lyapunov{} \ac{ADI} to non-zero initial values.
Furthermore, we generalize the permutation invariance of \ac{ADI} iterates
to a more general class of algorithms,
as well as the existence of a real-valued \ac{ADI} double-step for complex-conjugated shifts
to arbitrary linear systems.
}

\maketitle

%

\section{Intro}

We consider the numerical solution of the continuous-time \ac{ALE}
\begin{equation}%
\label{eq:intro}
  A X + X A^\HT = - GSG^\HT
\end{equation}
with large and sparse coefficient matrix $A\in\C^{n\times n}$
and a low-rank constant term comprised of the factors $G\in\C^{n\times g}$ and $S\in\C^{g\times g}$,
where $g\ll n$.
This type of equation arises in, \eg~optimal control and model order reduction.
We refer to~\cite{morAnt05,BenS13,Sim16} and the references therein for a more detailed introduction.

Given the low rank of the right-hand side of~\eqref{eq:intro},
the solution can, at least numerically, be well approximated by a low-rank factorization~\cite{Pen00b}.
We choose the symmetric indefinite factorization, \mbox{$X\approx \ZYZ^\HT$},
with tall and skinny $Z\in\C^{n\times z}$ and Hermitian $Y\in\C^{z\times z}$, $z\ll n$,
~\cite{BenLP08,LanMS15}.
For this kind of large-scale equations with low-rank solutions,
the most successful algorithms used recently are the Krylov subspace projection method,
\eg~\cite{morJaiK94,Sim07,PalS18,KreLMetal21},
as well as the \ac{ADI} method~\cite{Pen00b,LiW02,BenLP08,morBenKS13,Wac13}.
The latter will be the focus of this paper.

The low-rank \Lyapunov{} \ac{ADI} had been derived and analyzed
using Butcher tableaus for \acp{ODE}~\cite{BerF21,BerF21a},
as a Smith method~\cite{morGugL05},
as a Krylov method~\cite{WolP16,BerF20},
or by means of Cayley transformations~\cite{Kue16}.
We revise the derivation of the \ac{ADI} method using the language of linear splitting methods following~\cite[Section~5]{Sch22}.
The \ac{ADI} requires the solution of a linear system at every iteration,
which is the dominant part of the run-time of the algorithm.
Thus far,
for every derivation of the low-rank \Lyapunov{} \ac{ADI} \chdeleted{(LR-ADI)},
it has been mandatory to use an all zero initial guess, \mbox{$X_0=0$},~%
\cite{Pen00b,LiW02,BenLP08,BenLT09,BenKS12a,BenS13,morBenKS13,BenKS13,morBenKS13a,LanMS15}.
Our motivation is to reduce the number of \ac{ADI} iterations
by extending the method to arbitrary low-rank initial values, \mbox{$X_0\neq 0$}.
We generalize the iteration scheme of \citeauthorLiW~\cite{LiW02}
and the residual formulation of \citeauthorBenKS~\cite{morBenKS13}
to arbitrary low-rank initial values for the \chreplaced{low-rank \Lyapunov{} \ac{ADI}}{LR-ADI} method at the expense of an indefinite residual.

We investigate the performance of improved initial values using two outer
iterations in which \chreplaced{the low-rank \Lyapunov{} \ac{ADI}}{LR-ADI} is typically called.
First, we solve an \ac{ARE} with the Newton-ADI method;
see \eg{} \citeauthorBenLP~\cite{BenLP08}.
Every Newton step requires the solution of one \ac{ALE}.
Especially close to convergence of the Newton method, the solution of the
previous Newton step is expected to be a good candidate to start the ADI with.
\chdeleted{%
We investigate how much an LR-ADI warm-start of the Newton-ADI method
can close the performance gap compared to the \ac{RADI} method of \citeauthorBenBKetal~\cite{BenBKetal18}.
}

Second, we solve a \ac{DRE} with a first-order Rosenbrock method.
Each time step requires the solution of an \ac{ALE};
see \eg{} \citeauthorLanMS~\cite{LanMS15}.
Due to the smoothness of the solution, the solution at the previous time step
is a natural candidate to start the next ADI with.
Here, we are only concerned with an autonomous equation.
\chdeleted{and again we
investigate how much an LR-ADI warm-start in each time step can close the
performance gap to projection-based solvers for the \ac{DRE}\@.}
We expect the real
benefits for the non-autonomous case, where no alternative solvers exist, and
will separately report the results together with similar approaches for BDF methods for the non-autonomous \ac{DRE} along the lines
of~\cite{BarBS22, BarBSetal24}.

Throughout the paper, $\norm{\cdot}$ denotes the Frobenius norm.
$\Re(a+bi)=a$ and $\Im(a+bi)=b$ denote the real and complex part of a complex scalar (or matrix), respectively; $i=\sqrt{-1}$ and $a,b\in\R$.
We denote complex conjugation by $\conj{a+bi} = a-bi$.
$\mathbb{F}^{m\times n}$ denotes the space of $m$-by-$n$ matrices with entries in the field~$\mathbb{F}$.
Transposition of a real matrix is denoted by~${(\cdot)}^\TT : \R^{m\times n}\to\R^{n\times m}$,
Hermitian transposition of a complex matrix by~${(\cdot)}^\HT : \C^{m\times n}\to\C^{n\times m}$.
$I_q$ denotes the identity matrix of size~$q\in\N$.
Whenever we refer to general linear operators,
or the matrix dimensions are evident from the context,
we omit the subscript. To simplify notation in many places, we formulate the
concatenation of linear operators as multiplication, \ie \(A(B(X)) = AB(X) = ABX\).
Spectrum and spectral radius are denoted by $\Lambda(\cdot)$ and $\rho(\cdot)$, respectively.

The paper is structured as follows.
In \autoref{sec:splitting schemes},
we generalize the notion of splitting schemes and prove some of the known properties of the \chreplaced{low-rank \Lyapunov{} \ac{ADI}}{LR-ADI} in this more general context.
Afterwards, in \autoref{sec:adi},
we derive the ADI for arbitrary linear systems using the framework described in \autoref{sec:splitting schemes},
while generalizing some of the properties known in the \Lyapunov{} case.
\autoref{sec:lyapunov} specializes the ADI for \Lyapunov{} equations.
In \autoref{sec:applications} we present the two applications mentioned above,
as well as some numerical experiments.
We conclude the paper in \autoref{sec:conclusion}.

\section{Nonstationary Splitting Schemes}%
\label{sec:splitting schemes}
Many iterative methods solving $Ax=b$ can be written in a one-step \emph{splitting} form
\begin{equation}
  Mx^{k+1} = Nx^k + b
\end{equation}
where $M-N = A$ and systems $Mx = d$ are \enquote{easy} to solve~\cite[Section~11.2.3]{GolV13}.
These methods are \emph{consistent} by construction,
\ie every solution to $Ax=b$ is a fixed point of the iteration.
Conversely, every fixed point $x^{k+1} = x^k$ is a solution as well.
Assuming that $M^{-1}$ exists,
the method converges if the spectral radius $\rho(G) < 1$ for $G:=M^{-1}N$~\cite[Theorem~11.2.1]{GolV13}.
$G$ is called \emph{iteration matrix} of the scheme.

The structure above is called \emph{first normal form} of the method~\cite[Chapter~2]{Hac16}.
Consistency is equivalent to $M-N=A$.
The \emph{second normal form} of a consistent linear iteration method reads
\begin{equation}
  x^{k+1} = M^{-1} \big( (M-A) x^k + b \big) = x^k - M^{-1} (Ax^k - b)
  .
\end{equation}

Furthermore, a splitting method is called \emph{nonstationary}
\unskip\footnote{%
  Schulze~\cite{Sch22} called this family of algorithms \emph{parametrized} splitting schemes.
}
if the iteration matrix depends on the iteration,
\ie $A = M_{k} - N_{k}$ for $k\in\N$.
The normal forms are thus given by
\begin{subequations}%
\label{eq:normalform}
\begin{align}%
\label{eq:normalform:1}
  M_{k} x^{k+1} &= N_{k} x^k + b
  ,
  \\
\label{eq:normalform:2}
  x^{k+1} &= x^k - M_{k}^{-1} (Ax^k - b)
  .
\end{align}
\end{subequations}
Here, again, we assume that $M_k^{-1}$ exists. We will see in the later sections
that this is actually not a strong assumption in the context of this article.

\subsection{Commuting Splitting Schemes}
We call an operator split $A=M_{k}-N_{k}$
and the corresponding iterative method~\eqref{eq:normalform}
\emph{commuting} if $M_{k}, N_{k}$ commute, \ie
\begin{equation}\label{eq:commute}
  M_{k} N_{k} = N_{k} M_{k}
  \qquad
  \forall k \in\N.
\end{equation}

It is a well known fact that the iteration matrix $G_{k}$ determines
the evolution of the \emph{error} $e^k := x^k - x^*$, where $Ax^* = b$.
Due to the consistency of the method, $M_{k} x^* = N_{k} x^* + b$ holds.
Subtract this from the first normal form~\eqref{eq:normalform:1} to obtain
$M_{k} e^{k+1} = N_{k} e^k$ or, equivalently,
\begin{equation}%
\label{eq:adi:error-recursion}
  e^{k+1} = G_{k} e^k
  .
\end{equation}
That the same recursion holds for the \emph{residual} $r^k := Ax^k - b$, is an interesting observation for commuting operator splits.

\begin{proposition}[{Residual Recursion~\cite[Proposition~5.2]{Sch22}}]%
\label{thm:residual}%
\label{thm:r(r)}
  Let $A = M_{k} - N_{k}$ define a commuting nonstationary splitting method~\eqref{eq:normalform}.
  Then, the residual $r^k := Ax^k - b$ adheres to
  \begin{equation*}
    r^{k+1} = G_{k} r^k
  \end{equation*}
  where $G_{k} := M_{k}^{-1} N_{k}$ denotes the iteration matrix.
\end{proposition}
\begin{proof}
  Due to $A = M_{k} - N_{k}$ it is $A M_{k}^{-1} = I - N_{k} M_{k}^{-1}$.
  Multiplying~\eqref{eq:commute} with \(M_{k}^{-1}\) from both the
  left and the right, we find that \(M_k^{-1} N_k = N_k M_k^{-1}\).
  Thus, since \(M_{k}\) and \(N_{k}\) commute, so do $A$, $M_{k}^{-1}$, and $N_{k}$.
  Hence, by substituting the first normal form~\eqref{eq:normalform:1}
  into the definition of the residual,
  we obtain
  \begin{align*}
    r^{k+1}
    &= Ax^{k+1} - b \\
    &= A \big( M_k^{-1} (N_k x^k + b) \big) - b \\
    &= M_k^{-1} N_k \underbrace{
      \vphantom{M_k^{-1}}
      Ax^k
    }_{
      \vphantom{M_k^{-1}}
      r^k + b
    }
    + \underbrace{
      A M_k^{-1}
    }_{
      I\mathrlap{{} - N_k M_k^{-1}}
    } b - b \\
    &= M_k^{-1} N_k r^k + \underbrace{
      (M_k^{-1} N_k - N_k M_k^{-1})
    }_0 b
    .
    \myqedhere
  \end{align*}
\end{proof}

This observation allows us to formulate a variant of the second normal form~\eqref{eq:normalform:2}
that iterates the residual~$r^k$ and \emph{increment}~\eqIncFromRes{}
alongside the solution~$x^k$.
Although \autoref{thm:residual} may be seen as
a special case of Hackbusch~\cite[Exercise~2.15]{Hac16},
it allows us to give much simpler proofs of the known properties of the
\ac{ADI} method applied to \Lyapunov{} equations.
The details are covered in \autoref{sec:lyapunov}.

Using \autoref{thm:residual} together with commutation of \(M_{k}^{-1}\) and
\(N_{k}\), observed in its proof, as well as the definition of the increment~\eqIncFromRes, we conclude
the following corollary.
\begin{corollary}%
\label{thm:r(v)}
  The residual~$r^{k+1}$ only depends on
  the increment~$v^k$ leading to the iterate~$x^{k+1}$,
  that is
  \begin{equation*}
    \eqResFromInc
    .
  \end{equation*}
\end{corollary}

Combining this with the increment formula once more, gives:
\begin{corollary}%
\label{thm:v(v)}
  The increment~$v^k$ only depends on the previous increment~$v^{k-1}$,
  that is
  \begin{equation*}
    \eqIncFromInc
    .
  \end{equation*}
\end{corollary}

It may seem that \autoref{thm:r(v)} should always give the most efficient way
to update the residual.
However, \autoref{rem:lyapunov:complexity} provides a counter example.
At times,
\autoref{thm:r(r)} in combination with \eqIncFromRes{} yields a formula for
the residual~$r^{k+1}$ that is more efficient to evaluate.
In summary, this yields \autoref{alg:css}.

\begin{algorithm}[t]
  \caption{Commuting Splitting Scheme}\label{alg:css}
  \KwIn{%
    operator split $A=M_{k}-N_{k}$ with $M_{k}N_{k} = N_{k}M_{k}$ for $k\in\N$,
    initial guess $x^0$
  }
  \KwOut{%
    $v^0, v^1, \ldots$
    such that $x \approx x^0 + v^0 + v^1 + \ldots$ solves $Ax=b$
  }
  Assemble initial residual
  $r^0 \gets Ax^0 - b$\;
  \For{$k\in \{0, 1, \ldots\}$}{%
    Compute increment
    $v^k \gets - M_{k}^{-1} r^k$\;
    Update residual
    $r^{k+1} \gets -N_{k} v^k$
    (alternatively, via $r^{k+1} \gets M_{k}^{-1} N_{k} r^k$)
    \;
    \lIf{converged}{break}
  }
  Assemble solution
  $x \gets x^0 + v^0 + v^1 + \ldots$
  if needed\;
\end{algorithm}

\subsection{Fully Commuting Splitting Schemes}
We call an operator split $A = M_{k} - N_{k}$
and the corresponding iterative method~\eqref{eq:normalform}
\emph{fully commuting} if \(\forall i,j \in\N\) we have
\begin{equation}%
  \label{eq:fully commuting}
  M_i N_j = N_j M_i,
  \qquad
  M_i M_j = M_j M_i,
  \text{ and }\quad
  N_i N_j = N_j N_i.
\end{equation}
In the latter case,
the order in which the steps defined by $(M_0,N_0)$ to $(M_{k},N_{k})$ are applied does not matter:

\begin{proposition}[{Permutation Invariance~\cite[Proposition~5.1]{Sch22}}]%
\label{thm:permutation}
  Let $A = M_{k} - N_{k}$ be a fully commuting operator split.
  Let the initial iterate~$x^0$ be fixed.
  Then, the value of $x^{k+1}$ given by the nonstationary splitting method~\eqref{eq:normalform}
  does not depend on the order in which the steps are executed.
\end{proposition}
\begin{proof}
  Let $x^*$ denote the unique solution to $Ax=b$.
  Observe that $\forall i,j\in\N$,
  all iteration matrices $G_i := M_{i}^{-1} N_{i}$ and $G_j$ commute.
  Therefore, after applying $G_0$ through $G_k$,
  the error
  \begin{math}
    e^{k+1} = G_{k} G_{k-1} \cdots G_0 e^0
  \end{math}
  does not depend on the order of the steps,
  as does the iterate~$x^{k+1} = e^{k+1} + x^*$.
\end{proof}

\autoref{thm:permutation} generalizes \citeauthorLiW~\cite[Theorem~4.1]{LiW02}.

\section{ADI Method}%
\label{sec:adi}

The \ac{ADI} method has originally been described by \citeauthorPeaceman~\cite{PeaR55}.
Suppose \chreplaced{$A=\Ai+\Aii$}{$A=H+V$} for some linear operators $\Ai$ and $\Aii$
and the system to be solved is, again, $Ax=b$ for given $b$.
Select some parameters $\alpha_{k}, \beta_{k} \in\C$.
Using the shorthand notation
\begin{equation}%
\label{eq:adi:shorthand}
\begin{aligned}
  \Aip  &:= \Ai + \alpha_{k} I &
  \qquad\qquad\qquad
  \Aiip &:= \Aii + \beta_{k} I \\
  \Aim  &:= \Ai - \beta_{k} I &
  \Aiim &:= \Aii - \alpha_{k} I
  ,
\end{aligned}
\end{equation}
the \ac{ADI} method reads
\begin{equation}
  \begin{aligned}
    \Aip  x^{k+\frac{1}{2}} &= b - \Aiim x^k \\
    \Aiip x^{k+1}           &= b - \Aim x^{k+\frac{1}{2}}
    .
  \end{aligned}
\end{equation}
As a splitting scheme~\eqref{eq:normalform},
the \ac{ADI} method yields the consistent operator split
\begin{equation}\label{eq:ADIsplit}
  \begin{aligned}
    M_{k} &:= {(\alpha_{k} + \beta_{k})}^{-1} \Aip \Aiip \\
    N_{k} &:= {(\alpha_{k} + \beta_{k})}^{-1} \Aim \Aiim
    .
  \end{aligned}
\end{equation}
If $\Ai$ and $\Aii$ commute, \ie $\Ai\Aii=\Aii\Ai$,
so do all the shorthands~\eqref{eq:adi:shorthand} as well as the split
operators~\eqref{eq:ADIsplit}.
We summarize these findings in the following theorem.

\begin{theorem}[Fully Commuting Splitting Scheme]%
\label{thm:adi}
  The \ac{ADI} method is a fully commuting splitting scheme~\eqref{eq:fully commuting}
  assuming that $\Ai$ and $\Aii$ commute.
\end{theorem}

As we will motivate in the next section (\autoref{rem:lyapunov:eigenvalues})
the parameters may occur in conjugated pairs.
That is, $\alpha_{k+1} = \conj{\alpha_{k}}$ and $\beta_{k+1} = \conj{\beta_{k}}$.
It is therefore reasonable to look for a more efficient means to handle both steps $k$ and $k+1$ at once,
especially as the \ac{ADI} method preserves real iterates.

\begin{theorem}[Real-valued Double-Step]%
\label{thm:adi:2step}
  Let $\Ai$ and $\Aii$ be commuting \emph{real operators},
  \ie they map real elements onto real ones.
  Suppose that iterate~$x^k$ and residual~$r^k$ are real for some fixed $k\in\N$,
  and that the next \ac{ADI} parameters fulfill
  \begin{equation*}
    \alpha_{k} \alpha_{k+1},
    \quad
    \alpha_{k} + \alpha_{k+1},
    \quad
    \beta_{k} \beta_{k+1},
    \quad
    \beta_{k} + \beta_{k+1} \in\R,
  \end{equation*}
  for example, $\alpha_{k+1} = \conj{\alpha_{k}}$ and $\beta_{k}, \beta_{k+1} \in\R$.
  Then, two \ac{ADI} steps later, the iterate~$x^{k+2}$ and residual~$r^{k+2}$  are real again.
\end{theorem}
\begin{proof}
 We show constructively how to compute iterate~$x^{k+2}$ and residual~$r^{k+2}$ using real arithmetic.
  The residual~$r^{k+2} = G_{k+1} G_{k} r^k$ is real if
  the operator~$G_{k+1} G_{k}$ of the combined step is real.
  Let the scalar
  \begin{equation}
    \sigma_{k+1,k} :=
    (\alpha_{k+1} + \beta_{k+1})
    (\alpha_{k} + \beta_{k})
    \in\C\setminus\{0\}.
  \end{equation}
  Recall that $\Ai_k^+$ and $\Aii_{k+1}^+$ commute.
  We observe that the scaled
  \begin{equation}
    \sigma_{k+1,k}
    M_{k+1} M_{k}
    =
    \underbrace{
      (\Ai_{k+1}^+ \Ai_{k}^+)
    }_{\mathrlap{
      \hspace{-2ex}
      \begin{aligned}
        &= (\Ai+\alpha_{k+1}I) (\Ai+\alpha_{k} I) \\
        &= \Ai^2 + (\alpha_{k} + \alpha_{k+1}) \Ai + \alpha_{k} \alpha_{k+1} I
      \end{aligned}
    }}
    (\Aii_{k+1}^+ \Aii_{k}^+)
    ,
  \end{equation}
  is a real operator,
  as is the scaled $\sigma_{k+1,k} N_{k+1} N_{k}$ and, consequently, $G_{k+1} G_{k}$.
  This proves the desired property of the residual~$r^{k+2}$.

  Moreover,
  \autoref{thm:v(v)} together with the definition of the increment~\eqIncFromRes{} implies that
  \begin{math}
    M_{k+1} v^{k+1} = N_{k} v^k = -N_{k} M_{k}^{-1} r^k = -M_{k}^{-1} N_{k} r^k,
  \end{math}
  as $N_{k}$ and $M_{k}$ commute.
  Therefore,
  by multiplying $v^k$ and $v^{k+1}$ with $M_{k+1}M_{k}$,
  we conclude that
  the increment of the combined step
  \begin{math}
    x^{k+2} = x^k + v^k + v^{k+1}
  \end{math}
  is given by
  \begin{equation}%
    \label{eq:double-step}
    M_{k+1} M_{k}
    (v^k + v^{k+1})
    = - (M_{k+1} + N_{k}) r^k
    .
  \end{equation}
  Observe that the scaled
  \begin{subequations}
    \begin{align}
      \sigma_{k+1,k}
      (M_{k+1} + N_k)
      &=
      \begin{multlined}[t]
        (\alpha_k + \beta_k) (\Ai+\alpha_{k+1}I) (\Aii+\beta_{k+1}I) \\
        + (\alpha_{k+1} + \beta_{k+1}) (\Ai-\beta_{k}I) (\Aii-\alpha_{k}I)
      \end{multlined}
      \\
      &=
      \begin{multlined}[t]
        (\alpha_{k} + \alpha_{k+1} + \beta_{k} + \beta_{k+1}) \Ai\Aii \\
        + (\alpha_{k}\beta_{k+1} + \beta_{k}\beta_{k+1} - \alpha_{k}\alpha_{k+1} - \alpha_{k} \beta_{k+1}) \Ai \\
        + (\alpha_{k}\alpha_{k+1} + \alpha_{k+1}\beta_{k} - \alpha_{k+1}\beta_{k} - \beta_{k}\beta_{k+1}) \Aii \\
        + \big( (\alpha_{k} + \alpha_{k+1})\beta_{k}\beta_{k+1} + (\beta_{k} + \beta_{k+1})\alpha_{k}\alpha_{k+1} \big) I
      \end{multlined}
      \\
      &=
      \begin{multlined}[t]
        (\alpha_{k} + \alpha_{k+1} + \beta_{k} + \beta_{k+1}) \Ai\Aii
        + (\beta_{k}\beta_{k+1} - \alpha_{k}\alpha_{k+1}) (\Ai-\Aii) \\
        + \big( (\alpha_{k} + \alpha_{k+1})\beta_{k}\beta_{k+1} + \alpha_{k}\alpha_{k+1}(\beta_{k} + \beta_{k+1}) \big) I
      \end{multlined}
    \end{align}
  \end{subequations}
  is a real operator as well.
  Therefore,
  given that the residual~$r^k$ is real,
  equation~\eqref{eq:double-step} can be solved for the combined increment~$v^k + v^{k+1}$ using only real arithmetic.
  This proves the desired property for the iterate~$x^{k+2}$.
\end{proof}

\autoref{thm:adi:2step} generalizes~\cite{BenKS13}.
It motivates to look for an efficient formulation that handles complex shift parameters
when applying the \ac{ADI} method to \emph{any} linear system,
in particular, \emph{any} linear matrix equation.

\section{Low-rank \Lyapunov\ ADI}%
\label{sec:lyapunov}

In this section, we apply the \ac{ADI} method to an \ac{ALE} with low-rank right-hand side and solution.
We show how to derive several low-rank variants of the \ac{ADI} method~\cite{LiW02,BenLP08,BenLT09,BenK14,LanMS15},
while extending all these algorithms to non-zero initial values~$x^0 \neq 0$.
Throughout the section, we use the SVD-type low-rank factorization of,
\eg~\cite{BenLT09},
for the iterates~$x^k$, increments~$v^k$, residuals~$r^k$, and right-hand sides~$b$.

Consider the \ac{ALE}
\begin{equation}%
\label{eq:lyapunov}
  \Lyap(X) :=
  GSG^\HT + AX + XA^\HT = 0
\end{equation}
for square matrices $A, X\in\C^{n\times n}$, $G\in\C^{n\times g}$, and $S\in\C^{g\times g}$.
We assume that $\lambda_i + \conj{\lambda_j} \neq 0$ for any two eigenvalues~$\lambda_i,\lambda_j\in\C$ of $A$;
that is, equation~\eqref{eq:lyapunov} permits a unique solution;
see, \eg~\cite[Corollary~1.1.4]{AboFIetal03}
\chadded{for a system-theoretical perspective, or \cite{HorJ91} for a linear algebraic perspective.}
We are aiming to approximate the solution as
$X\approx \ZYZ^{\Htran}$, where $Z\in\C^{n\times z}$, $Y\in\C^{z\times z}$.
If the constant term has a low rank $g\ll n$,
we can expect the solution to have a low numerical rank as well~\cite{Pen00,AntSZ02,Gra04,Sab07,TruV07}.

In terms of the previous section,
$\Ai(U) := AU$ and $\Aii(U) := UA^{\Htran}$,
which obviously commute,
$(\Ai\Aii)(U) = AUA^{\Htran} = (\Aii\Ai)(U)$.
For generalized equations, the situation is not as simple.

\begin{remark}[Generalized Matrix Equations]%
\label{rem:generalized eqs}
  For a generalized \Lyapunov{} equation
  \begin{equation*}
    AXE^\TT + EXA^\TT = -W
  \end{equation*}
  we can, in general, not expect~$A$ and~$E$ to commute.
  For non-singular~$E$, however,
  it is still possible to symbolically transform the above equation
  into its simpler equivalent~\eqref{eq:lyapunov},
  in which case the restriction of~\autoref{thm:adi} becomes trivial.
  Then, however, it is mandatory to rephrase the actual steps of the \ac{ADI}
  individually along the lines of Saak~\cite[Section~5.2]{Saa09}, to avoid inversion of~$E$.
\end{remark}

For this reason,
we continue deriving the low-rank \ac{ADI} for standard equations~\eqref{eq:lyapunov}.
The overall \ac{ADI} operator split reads
\begin{equation}%
\label{eq:lyapunov:opsplit}
  \begin{aligned}
    M_{k}(U) &= {(\alpha_{k} + \beta_{k})}^{-1} (A + \alpha_{k} I) U {(A + \conj{\beta_{k}}  I)}^{\Htran},\\
    N_{k}(U) &= {(\alpha_{k} + \beta_{k})}^{-1} (A - \beta_{k}  I) U {(A - \conj{\alpha_{k}} I)}^{\Htran},
  \end{aligned}
\end{equation}
which is symmetry preserving if we choose $\beta_{k} := \conj{\alpha_{k}}$.
We further require that $-\alpha_{k}$ is not an eigenvalue of~$A$,
such that $M_{k}$ is invertible.
As long as the initial guess~$X_0 = Z_0 Y_0 Z_0^\HT$ is symmetric,
the same will hold for the residuals and increments.
The residual of a low-rank factorization can itself be expressed as a low-rank factorization,
\begin{math}
  \Lyap(\ZYZ^\HT) = \RTR^\HT
  ,
\end{math}
where
\begin{equation}%
\label{eq:lyapunov:residual:structure}
\begin{alignedat}{2}
  R &= \begin{bmatrix}
    G & Z & AZ
  \end{bmatrix}
  &&\in\C^{n\times(g+2z)}
  \\
  T &= \begin{bmatrix}
    S & \cdot & \cdot \\
    \cdot & \cdot & Y \\
    \cdot & Y & \cdot
  \end{bmatrix}
  &&\in\C^{(g+2z)\times(g+2z)}
\end{alignedat}
\end{equation}
have rank at most $g+2z\ll n$.
In particular, $\Lyap(0) = GSG^\HT$ is a rank-$g$ factorization, given both
\(G\) and \(S\) have full rank.

\begin{lemma}[Consistent \Lyapunov{} Residual]%
\label{thm:lyapunov:residual}
  There exists a factorization of the \Lyapunov{} \ac{ADI} residual~$r^k = R_{k} T_{} R_{k}^\HT$
  that is consistent with the initial residual~$r^0 = R_0 T_{} R_0^\HT$.
  That is, the inner factor $T$ does not depend on the iteration~$k\in\N$.
\end{lemma}
\begin{proof}
  Without loss of generality, suppose $r^k = R_{k} T_{k} R_{k}^\HT$ for $k\in\N$.
  Substitute the split operators~\eqref{eq:lyapunov:opsplit} into
  the residual update of \autoref{thm:r(r)}, \eqResFromRes,
  to obtain
  \begin{equation}
    (A + \alpha_{k} I) R_{k+1} T_{k+1} R_{k+1}^\HT {(A + \alpha_{k} I)}^\HT
    =
    (A - \conj{\alpha_{k}} I) R_{k} T_{k} R_{k}^\HT (A - \conj{\alpha_{k}} I)^\HT
    .
  \end{equation}
  By choosing the outer factor according to
  \begin{equation}%
  \label{eq:lyapunov:residual:naive}
    (A + \alpha_{k} I) R_{k+1} = (A - \conj{\alpha_{k}} I) R_{k}
    ,
  \end{equation}
  the inner factor does not need to change;
  that is, $T_{k+1} = T_{k} = T$.
\end{proof}

Moreover, we may utilize the definition of the increment~$v^k$
for a more efficient evaluation of the residual factors.
The formulation $M_{k}(v^k) = -r^k$ directly reveals a factorization $v^k = V_{k} \hat Y_{k} V_{k}^\HT$,
which we will now explain.
Using the \Lyapunov{} operator split~\eqref{eq:lyapunov:opsplit},
\begin{equation}
  {(2\Re\alpha_{k})}^{-1} (A+\alpha_{k} I) V_{k} \hat Y_{k} V_{k}^\HT {(A+\alpha_{k} I)}^\HT
  = - R_{k} T_{} R_{k}^\HT
  ,
\end{equation}
which is fulfilled for
\begin{equation}%
\label{eq:lyapunov:increments}
  (A+\alpha_{k} I) V_{k} = R_{k}
  ,
  \qquad
  \hat Y_{k} = -2\Re(\alpha_{k}) T
  .
\end{equation}
Therefore, by simple algebraic reformulations of the \Lyapunov{} residual formula~\eqref{eq:lyapunov:residual:naive},
we observe
\begin{equation}%
\label{eq:lyapunov:residual}
\begin{aligned}
  R_{k+1}
  &= {(A + \alpha_{k} I)}^{-1} (A - \conj{\alpha_{k}} I) R_{k} \\
  &= R_{k} - 2\Re(\alpha_{k}) {(A + \alpha_{k} I)}^{-1} R_{k} \\
  &= R_{k} - 2\Re(\alpha_{k}) V_{k}.
\end{aligned}
\end{equation}
This relation has previously been derived by
\citeauthorBenKS~\cite{morBenKS13} in the low-rank ADI context as well as
\citeauthorWolP~\cite{WolP16} interpreting the iteration as an implicit, in
general oblique, Krylov subspace projection method.

\begin{remark}[Efficient Residual Update]%
\label{rem:lyapunov:complexity}
  Had we instead used \autoref{thm:r(v)}, \eqResFromInc,
  to derive the residual factors,
  we would have obtained
  \begin{equation}%
  \label{eq:lyapunov:residual:eigen}
    R_{k+1} = (A-\conj{\alpha_{k}} I) V_{k}
    ,
  \end{equation}
  which can at best be evaluated with near-linear complexity if $A$ is sparse.
  In contrast, formula~\eqref{eq:lyapunov:residual}
  can be evaluated with linear complexity irrespective of $A$.
\end{remark}

\begin{remark}[Shifts and Eigenvalues]%
\label{rem:lyapunov:eigenvalues}
  Ideally, the residual~$r^{k+1}=R_{k+1} T_{} R_{k+1}^\HT=0$ should vanish.
  By formula~\eqref{eq:lyapunov:residual:eigen},
  that means $AV_{k} = \conj{\alpha_{k}} V_{k}$,
  which is only possible if $\conj{\alpha_{k}}$ is an eigenvalue of~$A$,
  and $V_{k}$ spans (a subspace of) the corresponding eigenspace.
  Relaxing this to a Ritz-Galerkin-type condition, multiplying with
  \(V_{k}^{\HT}\) from the left, and thus turning the eigenvalues into
  Ritz-values, motivates why both Penzl's heuristic shifts~\cite{Pen00b} and the
  self-generating projection shifts~\cite{BenKS14b} usually perform very well.
  The fact that, in the case of a real matrix, eigenvalues occur in conjugated pairs motivates \autoref{thm:adi:2step}.

  The problem with the so-called $V(u)$-shifts~\cite{BenKS14b}, however,
  is that $\alpha_k$ and $V_k$ are not chosen at the same time.
  Instead, $V_{k}$ is determined by $\alpha_k$ and the previous residual factor~$R_k$
  via formula~\eqref{eq:lyapunov:increments},
  while $\alpha_k$ is determined by some previous increments
  \begin{math}
    \tilde V := [
      V_{k-\ell-1} , V_{k-\ell-2} , \ldots , \linebreak[0] V_{k-\ell-u}
    ]
  \end{math}
  via
  \begin{math}
    \alpha_k \in\Lambda(\tilde V^\HT A\tilde V, \tilde V^\HT E\tilde V),
  \end{math}
  where $\ell\in\N$ denotes the number of steps taken since the last shift computation.
  Recall that the solution~$X$ has a low (numerical) rank.
  Therefore, we do not expect the space spanned by the increment~$V_k$ to deviate much from $\tilde V$,
  such that the offset~$\ell$ does not matter much.
\end{remark}

The arguments above lead to
the algorithm described by \citeauthorLanMS~%
\cite[Algorithm~3.1]{LanMS15}.
Recall, however, that our derivation does not require a zero matrix initial guess $X_0 = 0$.

Unfortunately, using the construction from the proof of \autoref{thm:adi:2step}
we obtain the algorithm described by \citeauthorBenLP~\cite[Algorithm~4]{BenLP08},
and not the more efficient formulation by \citeauthorBenKS~\cite[Algorithm~3]{BenKS13}.

However, the formulations by \citeauthorLanMS~\cite{LanMS15} to handle complex shifts
straight-forwardly apply to the case of a non-zero initial value in the \ac{ADI}
iteration.
We summarize the low-rank \Lyapunov{} \ac{ADI}
as derived by \autoref{thm:r(r)} and \autoref{rem:generalized eqs} in \autoref{alg:lyapunov}.
For the convergence criterion in \autoref{alg:converged} we use
\begin{equation}%
\label{eq:adi:convergence}
  \norm{r^k}
  = \norm{R_{k} T_{} R_{k}^\HT}
  \leq \begin{cases}
    \abstol_\text{ADI}
    \\
    \reltol_\text{ADI} \; \norm{GSG^\HT}
    =
    \reltol_\text{ADI} \; \norm{b}
    ,
  \end{cases}
\end{equation}
depending on whether $\abstol_\text{ADI}\in\R$
or $\reltol_\text{ADI}\in\R$ has been provided.
Either way, we evaluate all Frobenius norms of low-rank factorizations
using \autoref{alg:norm}~\cite{BenLP08,Pen00b}. Note that accumulation of \(Q\)
can be avoided~\cite{BenLP08},
\eg~using Householder~QR or (shifted) Cholesky~QR;
see~\cite{GolV13}.


\begin{algorithm}[t]
  \caption{Frobenius norm of a low-rank factorization}%
  \label{alg:norm}
  \KwIn{%
    matrices $Z\in\C^{n\times z}$ and $Y\in\C^{z\times z}$, $z\ll n$
  }
  \KwOut{%
    $\rho = \norm{\ZYZ^\HT}$
  }
  Compute \(R\) from the economy-size QR factorization $Z = QR$, \ie $R\in\C^{z\times z}$\;
  $\rho \gets \norm{R_{} Y_{} R^\HT}$\;
\end{algorithm}

\begin{algorithm}[t]
  \caption{Low-rank \Lyapunov{} \ac{ADI}}%
  \label{alg:lyapunov}
  \KwIn{%
    system matrices $A$, $E$, $G$, and $S$,
    initial value~$X_0 = Z_0 Y_0 Z_0^\HT$,
    parameters~$\{\alpha_0, \alpha_1, \ldots\}$
  }
  \KwOut{%
    $V_0, V_1, \ldots \in\C^{n\times m}$ and $T\in\C^{m\times m}$
    \chadded{comprising $Z$ and $Y$}
    such that $X \approx \ZYZ^\HT$ solves the
    \Lyapunov{} equation $AXE^\HT + EXA^\HT = -GSG^\HT$
  }
  Assemble initial residual factors:\linebreak
  \begin{math}
    R_0 \gets
    \begin{bmatrix}
      G & EZ_0 & AZ_0
    \end{bmatrix}
    ,
    \qquad
    T \gets
    \begin{bmatrix}
      S & \cdot & \cdot \\
      \cdot & \cdot & Y_0 \\
      \cdot & Y_0 & \cdot
    \end{bmatrix}
  \end{math}\;
  $k \gets 0$\;
  \Repeat{\label{alg:converged}converged}{%
    \eIf(\tcp*[f]{single step}){$\alpha_{k}$ is real}{%
      \label{alg:lyapunov:inc:real}
      Compute increment factor
      $V_{k} \gets {(A+\alpha_{k} E)}^{-1} R_{k}$\;
      Update residual factor
      $R_{k+1} \gets R_{k} - 2\Re(\alpha_{k}) E V_{k}$\;
      $k \gets k + 1$\;
    }(\tcp*[f]{double-step; $\alpha_{k+1} = \conj{\alpha_{k}}$ must hold}){
      \label{alg:lyapunov:inc:complex}
      Solve complex-valued system
      $\hat V_{k} \gets {(A+\alpha_{k} E)}^{-1} R_{k}$\;
      Compute increment factors:\linebreak%
      $\delta_{k} \gets \Re(\alpha_{k}) / \Im(\alpha_{k})$\linebreak%
      $V_{k} \gets \sqrt{2} \big( \Re(\hat V_{k}) + \delta_{k} \Im(\hat V_{k}) \big)$\linebreak%
      $V_{k+1} \gets \sqrt{2\delta_{k} + 2} \Im(\hat V_{k})$\;
      Update residual factor
      $R_{k+2} \gets R_{k} - 2\sqrt{2}\Re(\alpha_{k}) E V_{k}$\;
      $k \gets k + 2$\;
    }
  }
  Assemble solution factors, if needed:\linebreak
  \begin{math}
    Z \gets
    \begin{bmatrix}
      Z_0 & V_0 & V_1 & \ldots
    \end{bmatrix}
  \end{math}\linebreak
  \begin{math}
    Y \gets
    \operatorname{blockdiag}\big(
      Y_0,
      \enspace
      -2\Re(\alpha_0)T,
      \enspace
      -2\Re(\alpha_1)T,
      \enspace
      \ldots
    \big)
  \end{math}\;
\end{algorithm}

\begin{remark}[Other formulations of the \ac{ADI} iteration for Linear Matrix Equations]%
  Applying \autoref{thm:v(v)}, \eqIncFromInc,
  to the \Lyapunov{} operator split~\eqref{eq:lyapunov:opsplit}
  for \Cholesky-type low-rank iterates $X_{k} = Z_{k} Z_{k}^\HT$ and data $W=BB^\HT$
  yields the CF-ADI algorithm described by \citeauthorLiW~\cite{LiW02}.
  Applying the same reformulation as in equation~\eqref{eq:lyapunov:residual}
  leads to the LR-ADI as described by \citeauthorBenLP~\cite{BenLP08}.
  Recall, however, that our derivation does not require a zero matrix initial guess;
  that is, $X_0 = 0$ or $Z_0 = [\,]$.

  Applying \autoref{thm:v(v)}, \eqIncFromInc,
  to the Sylvester equation $AX-XB=W$
  leads to the fADI method described by \citeauthorBenLT~\cite{BenLT09}.
  Specializing \autoref{alg:css} to the Sylvester equation
  leads to the G-fADI described by \citeauthorBenK~\cite{BenK14},
  which iterates the residual alongside.
\end{remark}

\begin{remark}[Alternative Derivation]%
  Observe that $\Lyap(X_0 + \hat V) = 0$ is a \Lyapunov{} equation in the increment~$\hat V$,
  \begin{equation*}
    \Lyap(X_0 + \hat V)
    = RTR^\HT + A \hat V + \hat V A^\HT
    = 0.
  \end{equation*}
  where $R$ and $T$ are the \Lyapunov{} residual factors~\eqref{eq:lyapunov:residual:structure}.
  Applying the \ac{ADI} to solve for~$\hat V$ and starting with rank-zero $\hat V_0 = 0$,
  is equivalent to
  applying the \ac{ADI} to solve for~$X_0+\hat V$ and starting with $X_0$.
\end{remark}

\section{Applications and Numerical Experiments}%
\label{sec:applications}

In this section,
we recall two iterative algorithms that solve an \ac{ALE} at every step.
If we solve these \ac{ALE}s with another iterative method,
like the \ac{ADI},
it is natural to use the previous outer iterate as the initial guess for the inner method.

More specifically,
let $X_\ell$ denote the outer iterates, $\ell\in\N$,
and let $X_{\ell+1,k}$ denote the inner iterates, $k\in\N$,
whose final iterate will become $X_{\ell+1}$.
Then, it is natural to choose $X_{\ell+1,0} = X_\ell$ instead of, e.g., $X_{\ell+1,0} = 0$.

All computations in this section have been performed on an Intel~Xeon Silver~4110.
We evaluated the following shift strategies:
\begin{itemize}
  \item
    $\heuristic(10, 10, 10)$:
    Penzl's $\heuristic(l_0, k_+, k_-)$ strategy,
    which first computes eigenvalues based on
    $k_+$ Arnoldi iterations of $E^{-1}A$ and
    $k_-$ Arnoldi iterations of $A^{-1}E$,
    and selects $l_0$ of them based on a greedy heuristic~\cite{Pen00b}.
    If more than $l_0$ shifts are needed, we repeat these $l_0$-many selected
    shifts cyclically.
  \item
    $\heuristic(20, 30, 30)$
  \item
    \label{item:shifts:proj:heur}
    $\xheurproj{2}$:
    One of the self-generating projection strategies,
    which is called $V(u)$-shifts in~\cite[Section~5.3.1]{Kue16}.
    In our case $u=2$, to properly account for a potential \ac{ADI} double-step.
    Here, we order these Ritz values based on Penzl's heuristic.
  \item
    $\xdecrproj{2}$:
    The same $V(2)$-shifts ordered by decreasing real part.
    In the case of real-valued shifts,
    the shifts are equivalently ordered by increasing magnitude,
    since all such shifts must have a negative real part.
  \item
    $\xincrproj{2}$:
    The same $V(2)$-shifts ordered by increasing real part.
\end{itemize}
In general,
we expect fewer \ac{ADI} iterations for larger $l_0\in\N$ in the case of $\heuristic(l_0, k_+, k_-)$ shifts,
and even fewer iterations for the $\projection$ shifts.
In that notion, we call shifts to be better or worse than others.

\begin{remark}[Accuracy of Ritz Values]
  It is critical that Penzl's shifts are computed precisely as described
  in~\cite{Pen00b}, since it is by no means clear that the largest or smallest
  eigenvalues and corresponding eigenspaces are the most relevant for the
  solution process.
  For example, using the~$k_+$ largest and~$k_-$ smallest eigenvalues of $E^{-1}A$
  as computed by ArnoldiMethod.jl~\cite{StoN21} (version 0.2.0) compared to
  Penzl's procedure, limiting to only a few Arnoldi steps and Ritz values, being
  comparably rough approximations of the eigenvalues, leads to \qtyrange{5}{10}{\timesunit} more \ac{ADI} iterations.
  These values are too accurate and ignore relevant eigenspaces still
  covered by Penzl's Ritz values. When computing highly accurate eigenvalues,
  choosing those from the dominant poles of the underlying dynamical
  system, for which the \ac{ALE} is solved, can be beneficial~\cite[Section~8.1]{Saa09}.
\end{remark}

\begin{remark}[Order of Shifts and Eigenvalues]%
\label{rem:shift order}
  \citeauthorBenKS\ require that complex conjugated pairs remain adjacent,
  but do otherwise not specify the order in which the $V$-shifts shall be used~\cite{BenKS14b}.
  As we will see later,
  the \ac{ADI} can be quite susceptible to the order of the shifts.
  Given that, for example,
  MATLAB's \verb!eig! does not guarantee a particular order
  but appears to keep complex conjugated pairs adjacent,
  while Julia's \verb!eigvals! sorts the eigenvalues
  lexicographically by real and imaginary parts (similar to the $\xincrproj{2}$ shifts),
  the unaware user may observe drastically different convergence behaviors.
  Try, for example,
  \begin{equation*}
    A = \operatorname{blockdiag}\left(
      \begin{bmatrix}
        -3 & 4 \\
        -4 & 3
      \end{bmatrix}
      ,
      \begin{bmatrix}
        -4 & 3 \\
        -3 & 4
      \end{bmatrix}
      ,
      \begin{bmatrix}
        -3 & 2 \\
        -2 & 3
      \end{bmatrix}
      , -5, -3
    \right)
    .
  \end{equation*}
  Julia's order is problematic insofar as complex conjugated pairs are not adjacent,
  which is required for an \ac{ADI} double-step.
  Therefore, when sorting the eigenvalues based on their real part,
  we also sort by the absolute value of their imaginary parts,
  in order to keep complex conjugated pairs adjacent.
  Penzl's heuristic ensures this property by itself.
\end{remark}

We conclude this section's introduction by mentioning some implementation details.
The examples to follow will all lead to real-valued system matrices.
We therefore expect a real-valued solution~$ZYZ^\TT \in\R^{n\times n}$ as well;
see \autoref{thm:adi:2step}.
All Frobenius norms of $\ZYZ^\TT$ factorizations are evaluated using \autoref{alg:norm}.
The systems arising in the \ac{ADI} increment formulas~\eqref{eq:lyapunov:increments},
lines~\ref{alg:lyapunov:inc:real} and~\ref{alg:lyapunov:inc:complex} of \autoref{alg:lyapunov},
may have a low-rank updated structure,
\begin{equation}%
\label{eq:application:increment}
  \underbrace{A + \alpha_k E}_{\text{sparse}}
  +
  \underbrace{\vphantom{\alpha_kE}UV^\TT}_{\text{low-rank}}
  \in\C^{n\times n}
  ,
\end{equation}
where $U$ and $V$ are real and have few columns.
The only complex contribution will be from the \ac{ADI} shifts~$\alpha_k\in\C$.
In this case,
we apply the Sherman-Morrison-Woodbury formula;
\eg~\cite[Section~2.1.4]{GolV13}.
\unskip\footnote{%
  For more general examples,
  the structure of the low-rank part in~\eqref{eq:application:increment}
  may be $UDV^\TT$ for some $D\neq I$.
  In that case,
  the center terms~$S_\ell$ of the \Lyapunov{} equations to come (equations~\eqref{eq:newton:step} and~\eqref{eq:rosenbrock:step})
  need to be modified slightly,
  and the Sherman-Morrison-Woodbury formula requires a term~$D$.
  For ease of notation,
  we stick to the simpler case~\eqref{eq:application:increment}.
}

We perform a column compression of any factored representation $\ZYZ^\TT \in\R^{n\times n}$,
as described by \citeauthorLanMS~\cite{LanMS15},
every 10 low-rank additions/subtractions,
\begin{equation}
  Z_1Y_1Z_1^\TT \pm Z_2Y_2Z_2^\TT =
  \begin{bmatrix}
    Z_1 & Z_2
  \end{bmatrix}
  \underbrace{%
  \begin{bmatrix}
    Y_1 \\
    & \pm Y_2
  \end{bmatrix}
  }_{Y\mathrlap{\in\R^{k\times k}}}
  \begin{bmatrix}
    Z_1 & Z_2
  \end{bmatrix}^\TT
  ,
\end{equation}
or once the inner dimension reaches 50\% of the outer dimension; $k \geq 0.5n$.
The eigenvalues of the resulting inner matrix~$\hat Y$ are truncated such that
\begin{equation}
  \Lambda(\hat Y) =
  \big\{
    \lambda\in\Lambda(Y) :
    \abs{\lambda} \geq \max\{1,\rho(Y)\} \cdot k \cdot \umach
  \big\}
  ,
\end{equation}
where $\Lambda(\cdot)$ and $\rho(\cdot)$ denote the spectrum and spectral radius, respectively.
We further compress every low-rank \Lyapunov{} right-hand side~$W$
before solving $AX+XA^\TT = -W$,
unless mentioned otherwise.

\subsection{Algebraic \Riccati{} Equation}

We apply the Newton-Kleinman method to the \acf{ARE}
\begin{equation}
  \Ricc(X) :=
  C^\TT C + A^\TT X E + E^\TT X A - E^\TT X BB^\TT X E = 0,
  \label{eq:are}
\end{equation}
with matrices
\begin{equation}
  E, A \in\R^{n\times n},
  \quad
  B \in\R^{n\times m},
  \quad
  C \in\R^{q\times n},
\end{equation}
and $m,q\ll n$.
Due to the low rank of $C^\TT C$ and $BB^\TT$,
we can expect the solution to have a low numerical rank;
see Benner and Bujanovi\'{c}~\cite{BenB16}.
Hence,
we factorize the Newton iterates according to $X_\ell = Z_\ell Y_\ell Z_\ell^\TT$
with $Z_\ell \in\R^{n\times z_\ell}$ and $Y_\ell\in\R^{z_\ell\times z_\ell}$.
Adapted from \citeauthorBenLP~\cite{BenLP08},
the \(\ell\)\/th step~($\ell\in\N$) of the Kleinman formulation reads
\begin{equation}
  0 = \Lyap_\ell(X_{\ell+1}) :=
  G_\ell S_\ell G_\ell^\TT +
  A_\ell^\TT X_{\ell+1} E + E^\TT X_{\ell+1} A_\ell,
\label{eq:newton:step}
\end{equation}
where
\begin{equation}
\begin{alignedat}{2}
  A_\ell &= A - BB^\TT X_\ell E
         &&\in\R^{n\times n},
  \\
  G_\ell &= \begin{bmatrix}
    C^\TT & E^\TT X_\ell B
  \end{bmatrix}
         &&\in\R^{n\times(q+m)},
  \\
  S_\ell &=
  I
         &&\in\R^{(q+m)\times(q+m)}.
\end{alignedat}
\label{eq:newton:step:matrices}
\end{equation}
Note that $A_\ell$ has the structure of a sparse matrix updated by a low-rank
matrix product, using \(U=B\) and \(V^{\TT}=B^{\TT} X_{\ell} E\).
Let $X_{\ell+1,k}$ for $k\in\N$ denote the \ac{ADI} iterates during the solution of~\eqref{eq:newton:step}.
\chreplaced{%
Note further that although the value of~$G_\ell$ depends on the current iterate~$X_\ell$,
the dimensions of~$G_\ell$ are constant.
}{%
The dimension of the right-hand side factor~$G_\ell$ does not depend on the current iterate~$X_\ell$.
}
\unskip\footnote{
  Except for an outer initial value~$X_0=0$,
  which would allow $G_0 = C^\TT$ which has only $q$ columns.
  For simplicity, we did not implement this,
  nor do we perform a column-compression on the right-hand side.
}
Therefore, the \ac{ADI} started with $X_{\ell+1,0} = 0$ will have an advantage,
as its \Lyapunov{} residuals~$\Lyap_\ell(X_{\ell+1,k})$ will always have (at most) $q+m$ columns,
which determines the cost to solve the linear systems defining the \ac{ADI} increments.
The new \ac{ADI} with $X_{\ell+1,0}=X_\ell=Z_\ell Y_\ell Z_\ell^\TT$ having inner dimension $z_\ell$ will lead to
\Lyapunov{} residuals having up to $(q+m)+2z_\ell$ columns;
see formula~\eqref{eq:lyapunov:residual:structure}.
We do not yet use the common $E^\HT Z_\ell$ block in the outer low-rank factor,
which would allow factorization having inner dimension $q+2z_\ell < q+m+2z_\ell$ and, thus,
slightly reduced compression times.

The \Riccati{} residual of a low-rank factorization can naively be written in a low-rank form as
\begin{math}
  \Ricc(Z_\ell Y_\ell Z_\ell^\TT) = \RTR^\TT
  ,
\end{math}
where
\begin{equation}
  \begin{aligned}
    R &= \begin{bmatrix}
      C^\TT & A^\TT Z_\ell & E^\TT Z_\ell
    \end{bmatrix}
      &&\in\R^{n\times(q+2z_\ell)}
    \\
    T &= \begin{bmatrix}
      I & \cdot & \cdot \\
      \cdot & \cdot & Y_\ell \\
      \cdot & Y_\ell & -Y_\ell Z_\ell^\TT BB^\TT Z_\ell Y_\ell
    \end{bmatrix}
      &&\in\R^{(q+2z_\ell)\times(q+2z_\ell)}
    ,
  \end{aligned}
  \label{eq:are:naive residual}
\end{equation}
compare, e.g.,~\cite{BenLP08}.
We consider the Newton method to have converged once
\begin{equation}
  \norm{\Ricc(X_{\ell+1})} < \reltol_\mathrm{Newton} \norm{C^\TT C}
\end{equation}
for some user-provided $\reltol_\mathrm{Newton}$.
Next,
we describe how to configure the \ac{ADI} convergence criterion~\eqref{eq:adi:convergence}
based on three different variations of the Newton method.
First, for the classical Newton method we use
\begin{subequations}
\begin{equation}
  \reltol_\textup{\ac{ADI}} = \reltol_\textup{Newton} / 10
  .
\label{eq:newton:classical}
\end{equation}
Second,
to speed up the early Newton iterations,
we use an inexact Newton method
with variable forcing term~$\eta\in\R$
described by \citeauthorDemES~\cite{DemES82},
resulting in
\begin{equation}
  \abstol_\textup{\ac{ADI}} = \eta\norm{\Ricc(X_\ell)}
  .
\label{eq:newton:inexact}
\end{equation}
Third,
to speed up the later Newton steps with already small $\norm{\Ricc(X_\ell)}$,
we switch back to the classical Newton method
if the inexact bound~\eqref{eq:newton:inexact} becomes smaller than the classical bound~\eqref{eq:newton:classical}.
This ``hybrid'' Newton method results in
\begin{equation}
  \abstol_\textup{\ac{ADI}} =
  \max\big\{ \eta\norm{\Ricc(X_\ell)}, \enspace \tfrac{1}{10}\reltol_\textup{Newton}\norm{\Lyap_\ell(0)} \big\}
  .
\label{eq:newton:hybrid}
\end{equation}
\end{subequations}

Furthermore,
to keep the potential overshoot of the \Riccati{} residual in control,
which refers to an increasing \Riccati{} residual, as usually observed in the first Newton iterations,
we employ an Armijo line search.
\unskip\footnote{%
  Our implementation is rather naive in that it does not utilize the connection between $\Ricc(X)$
  and $\Lyap_\ell(X)$ described in~\cite[Equation~(5.2b)]{BenHSetal16}.
  However, the overall runtime cost of the line search was negligible (less than 1\% of the total runtime).
}
As described by \citeauthorBenHSetal~\cite{BenHSetal16},
the line search is applied if $\norm{\Ricc(\hat X_{\ell+1})} > 0.9 \norm{\Ricc(X_\ell)}$,
where $\hat X_{\ell+1}$ denotes the solution to the classical Newton step~\eqref{eq:newton:step}.

\begin{example}%
\label{ex:are}
We apply the procedure described above to the Steel Profile benchmark~\cite{morwiki_steel,BenS05b},
which is based on a semi-discretized heat transfer problem.
The corresponding matrices~\eqref{eq:dre:matrices}
have \mbox{$m=7$} inputs, \mbox{$q=6$} outputs,
and are available for several sizes~$n$.
For this work, we focus on the configuration \mbox{$n=5177$}.
However, the input matrix $B = 1000\hat B$ is a scalar multiple of the original
Steel Profile matrix~$\hat B$.
The factor $1000$ controls the relative weighting of control and
output costs of the corresponding linear quadratic regulator problem. In this
setting it corresponds to a factor \(10^{-6}\) on the control costs, decreasing
the regularity of the optimal control problem and, thus, making the \ac{ARE}
harder to solve. Consequently, we observe a higher overshoot of the \Riccati{}
residual norm and more Newton steps.
We use the quadratic forcing term
\begin{math}
  \eta = \min\{ 0.1, 0.9 \norm{\Ricc(X_\ell)} \}
\end{math}
also used by \citeauthorBenHSetal~\cite{BenHSetal16}.
The Newton tolerance is
\begin{math}
  \reltol_\textup{Newton} = 10^{-10}
  .
\end{math}
\end{example}

\afterpage{%
\begin{landscape}
\begin{table}[h]
  \centering
  \caption{%
    Total number of Newton iterations~$\ell_{\max}$
    and \ac{ADI} iterations~$k_\mathrm{total}$ performed,
    \chadded{the number of columns~$z_{\ell_{\max}}$ of the final outer solution factor (its rank),}
    as well as several run-time metrics
    to solve \ac{ARE}~\eqref{eq:are} arising from \autoref{ex:are}.
    All timings are in seconds.
    The \ac{ADI} is started from a zero value (old)
    or with the previous Newton iterate (new).
    A dash (--) indicates that the method did not converge.
    A speedup \num{>1} indicates that our new \ac{ADI} is faster.
  }%
  \label{tab:are}
  \resizebox{\linewidth}{!}{%
  \begin{tabular}{%
    ll
    S[table-format=4]S[table-format=4]
    S[table-format=2,table-column-width=8mm]S[table-format=2,table-column-width=8mm]
    S[table-format=3]S[table-format=3]
    S[table-format=2.1]S[table-format=2.1]
    S[table-format=2.1]S[table-format=2.1]
    S[table-format=2.1]S[table-format=2.1]
    S[table-format=2.1]S[table-format=2.1]
    S[round-precision=2,table-format=1.2]
  }
    \toprule
    {\multirow{2}{*}[-2pt]{Newton method}} &
    {\multirow{2}{*}[-2pt]{\ac{ADI} shifts}} &
    \multicolumn{2}{c}{\#ADI steps} &
    \multicolumn{2}{c}{\#Newton steps} &
    \multicolumn{2}{c}{Rank} &
    \multicolumn{2}{c}{$t_\mathrm{compress}$} &
    \multicolumn{2}{c}{$t_\mathrm{shifts}$} &
    \multicolumn{2}{c}{$t_\mathrm{solve}$} &
    \multicolumn{2}{c}{$t_\mathrm{total}$} &
    {\multirow{2}{*}[-2pt]{Speedup}}
    \\
    \cmidrule(lr){3-4}
    \cmidrule(lr){5-6}
    \cmidrule(lr){7-8}
    \cmidrule(lr){9-10}
    \cmidrule(lr){11-12}
    \cmidrule(lr){13-14}
    \cmidrule(lr){15-16}
&
&
    {old} & {new} &
    {old} & {new} &
    {old} & {new} &
    {old} & {new} &
    {old} & {new} &
    {old} & {new} &
    {old} & {new} &
    \\\midrule
    classical & $\heuristic(10, 10, 10)$ & 1273 & 756 & 15 & 15 & 146 & 151 & 10.6 & 18.4 & 3.7 & 3.2 & 31.8 & 24.6 & 52.0 & 57.9 & 0.90 \\
classical & $\heuristic(20, 30, 30)$ & 740 & 514 & 15 & 15 & 147 & 151 & 5.3 & 12.6 & 9.7 & 9.5 & 18.1 & 16.2 & 36.9 & 46.9 & 0.79 \\
classical & $\xdecrproj{2}$ & 941 & 771 & 15 & 15 & 146 & 151 & 7.5 & 24.0 & 0.6 & 0.8 & 23.0 & 28.3 & 36.6 & 67.4 & 0.54 \\
classical & $\xheurproj{2}$ & 748 & 805 & 15 & 15 & 148 & 151 & 5.8 & 19.5 & 0.4 & 1.2 & 17.0 & 26.5 & 29.1 & 58.2 & 0.50 \\
classical & $\xincrproj{2}$ & 963 & 1525 & 15 & 15 & 145 & 151 & 8.1 & 35.4 & 0.7 & 0.9 & 23.8 & 45.6 & 37.9 & 101.9 & 0.37 \\
classical w/ line search & $\heuristic(10, 10, 10)$ & 1175 & 676 & 12 & 12 & 146 & 151 & 9.0 & 16.0 & 2.9 & 2.1 & 23.9 & 16.7 & 40.5 & 42.2 & 0.96 \\
classical w/ line search & $\heuristic(20, 30, 30)$ & 595 & 367 & 12 & 12 & 147 & 151 & 4.2 & 8.3 & 7.6 & 6.9 & 14.0 & 12.2 & 30.2 & 35.4 & 0.85 \\
classical w/ line search & $\xdecrproj{2}$ & 718 & 595 & 12 & 12 & 146 & 151 & 6.2 & 19.2 & 0.4 & 0.6 & 16.3 & 23.3 & 28.1 & 55.7 & 0.51 \\
classical w/ line search & $\xheurproj{2}$ & 604 & 650 & 12 & 12 & 148 & 151 & 5.0 & 16.9 & 0.6 & 0.7 & 13.9 & 20.4 & 24.2 & 49.1 & 0.49 \\
classical w/ line search & $\xincrproj{2}$ & 771 & 1302 & 12 & 12 & 145 & 151 & 7.0 & 31.2 & 0.3 & 0.8 & 18.7 & 41.0 & 30.7 & 90.7 & 0.34 \\
inexact & $\heuristic(10, 10, 10)$ & 636 & 161 & 15 & 15 & 146 & 151 & 4.4 & 7.8 & 3.7 & 3.6 & 13.5 & 6.3 & 26.2 & 22.7 & 1.15 \\
inexact & $\heuristic(20, 30, 30)$ & 419 & 185 & 15 & 15 & 147 & 151 & 3.6 & 7.0 & 8.5 & 8.7 & 9.3 & 7.2 & 23.8 & 28.9 & 0.82 \\
inexact & $\xdecrproj{2}$ & {--} & 373 & {--} & 11 & {--} & 148 & {--} & 14.1 & {--} & 0.4 & {--} & 11.8 & {--} & 33.6 & {--} \\
inexact & $\xheurproj{2}$ & {--} & {--} & {--} & {--} & {--} & {--} & {--} & {--} & {--} & {--} & {--} & {--} & {--} & {--} & {--} \\
inexact & $\xincrproj{2}$ & {--} & {--} & {--} & {--} & {--} & {--} & {--} & {--} & {--} & {--} & {--} & {--} & {--} & {--} & {--} \\
inexact w/ line search & $\heuristic(10, 10, 10)$ & 589 & 142 & 12 & 12 & 146 & 151 & 4.7 & 6.5 & 2.6 & 2.6 & 13.0 & 5.2 & 24.1 & 18.9 & 1.27 \\
inexact w/ line search & $\heuristic(20, 30, 30)$ & 404 & 162 & 12 & 12 & 147 & 151 & 3.0 & 5.6 & 8.3 & 7.4 & 8.1 & 5.5 & 21.5 & 23.3 & 0.92 \\
inexact w/ line search & $\xdecrproj{2}$ & {--} & {--} & {--} & {--} & {--} & {--} & {--} & {--} & {--} & {--} & {--} & {--} & {--} & {--} & {--} \\
inexact w/ line search & $\xheurproj{2}$ & {--} & {--} & {--} & {--} & {--} & {--} & {--} & {--} & {--} & {--} & {--} & {--} & {--} & {--} & {--} \\
inexact w/ line search & $\xincrproj{2}$ & {--} & {--} & {--} & {--} & {--} & {--} & {--} & {--} & {--} & {--} & {--} & {--} & {--} & {--} & {--} \\
hybrid & $\heuristic(10, 10, 10)$ & 557 & 141 & 15 & 15 & 146 & 151 & 4.1 & 6.6 & 3.7 & 2.6 & 12.2 & 4.9 & 23.7 & 18.4 & 1.29 \\
hybrid & $\heuristic(20, 30, 30)$ & 373 & 155 & 15 & 15 & 147 & 151 & 2.2 & 6.4 & 10.3 & 8.8 & 7.4 & 5.4 & 22.8 & 27.7 & 0.82 \\
hybrid & $\xdecrproj{2}$ & {--} & 304 & {--} & 11 & {--} & 148 & {--} & 12.2 & {--} & 0.4 & {--} & 10.6 & {--} & 29.3 & {--} \\
hybrid & $\xheurproj{2}$ & {--} & {--} & {--} & {--} & {--} & {--} & {--} & {--} & {--} & {--} & {--} & {--} & {--} & {--} & {--} \\
hybrid & $\xincrproj{2}$ & {--} & {--} & {--} & {--} & {--} & {--} & {--} & {--} & {--} & {--} & {--} & {--} & {--} & {--} & {--} \\
hybrid w/ line search & $\heuristic(10, 10, 10)$ & 516 & 121 & 12 & 12 & 146 & 151 & 3.6 & 5.0 & 2.1 & 2.1 & 9.6 & 3.6 & 18.3 & 15.7 & 1.17 \\
hybrid w/ line search & $\heuristic(20, 30, 30)$ & 348 & 132 & 12 & 12 & 147 & 151 & 2.6 & 5.0 & 7.0 & 7.5 & 7.7 & 4.9 & 19.7 & 22.0 & 0.90 \\
hybrid w/ line search & $\xdecrproj{2}$ & {--} & {--} & {--} & {--} & {--} & {--} & {--} & {--} & {--} & {--} & {--} & {--} & {--} & {--} & {--} \\
hybrid w/ line search & $\xheurproj{2}$ & {--} & {--} & {--} & {--} & {--} & {--} & {--} & {--} & {--} & {--} & {--} & {--} & {--} & {--} & {--} \\
hybrid w/ line search & $\xincrproj{2}$ & {--} & {--} & {--} & {--} & {--} & {--} & {--} & {--} & {--} & {--} & {--} & {--} & {--} & {--} & {--} \\

    \bottomrule
  \end{tabular}%
  }
\end{table}
\end{landscape}
}

\autoref{tab:are} gives an overview on the number of \ac{ADI} and Newton iterations required
to solve the \ac{ARE}~\eqref{eq:are} arising from \autoref{ex:are},
as well as the accumulated run-time of the $\ZYZ^\TT$ columns compression, $t_\mathrm{compress}$,
the time to compute the \ac{ADI} shifts, $t_\mathrm{shifts}$,
the time of all the linear solves of \autoref{alg:lyapunov}, $t_\mathrm{solve}$,
as well as the total run-time of the Newton method.
The \ac{ADI} shift strategies are described at the beginning of this section.
We observe that enabling the line search reduces the number of Newton steps,
but only mildly reduces the number of \ac{ADI} steps,
as mentioned in~\cite{BenHSetal16}.
Switching to the inexact or hybrid Newton method has a much larger effect on the number of \ac{ADI} steps.
Replacing the zero initial \ac{ADI} value by the previous Newton iterate
further reduces the number of \ac{ADI} steps by up to \qty{4}{\timesunit} (\num{516} vs \num{121}).
This reduction is strongest for the simplest shift strategy, $\heuristic(10, 10, 10)$.
The old \ac{ADI} outperforms the new one in most configurations of the Newton method,
and its run-time is dominated by~$t_\mathrm{solve}$.
Only for the inexact and hybrid Newton method with $\heuristic(10, 10, 10)$ shifts,
our new \ac{ADI} has a slight run-time advantage.
In these cases, however, $t_\mathrm{compress}$ is the most expensive part of the
algorithm,
indicating further potential for improvements especially for our new \ac{ADI}, when only bad shifts are known.

We observe a certain dependence of the new \ac{ADI} on the order of the (projection) shifts.
By merely changing from the decreasing to the increasing order (see \autoref{rem:shift order}),
the number of \ac{ADI} iterations more than doubles (\num{595} vs \num{1302}).
The large number of \ac{ADI} iterations for $\xincrproj{2}$ shifts is somewhat reasonable;
refer to \autoref{sec:app:shift order} for an explanation.
Meanwhile, the old \ac{ADI} seems barely affected.
However, most configurations using inexact or hybrid Newton with projection shifts did not converge.
Only the new \ac{ADI} with $\xdecrproj{2}$ shifts made the inexact methods without line search converge.
Further inspection is needed to understand this phenomenon.

In terms of the number of \ac{ADI} iterations,
the optimal configuration for the old \ac{ADI} is the hybrid Newton method with line search and $\heuristic(20, 30, 30)$ shifts.
The new \ac{ADI} performed best for the same Newton method and the simpler $\heuristic(10, 10, 10)$ shifts.
Our present implementation of the heuristic shift strategy is not yet optimized,
which skews the total run-time of the old \ac{ADI} in favor of the simpler shifts.
Thus, currently, we observe a \qty{3}{\timesunit} reduction in the number of \ac{ADI} steps (\num{348} vs \num{121}),
and a modest speedup of about \qty{17}{\%} in favor of the new \ac{ADI} (\qty{18.3}{\second} vs \qty{15.7}{\second})
when compared to the old \ac{ADI} with the lowest number of \ac{ADI} steps and run-time, respectively.
However,
for any configuration other than the inexact Newton methods with the simplest shift strategy,
the old \ac{ADI} outperforms our new \ac{ADI}.
More research is needed in how to select the optimal Newton configuration and \ac{ADI} shifts.

For the remainder of this subsection,
we focus on configuration with the lowest run-time for either \ac{ADI},
namely the hybrid Newton method with line search using $\heuristic(10, 10, 10)$ shifts.
\autoref{fig:are:macro} shows that the new \ac{ADI} requires much fewer iterations
per Newton step; as expected.
We further see that the number of old \ac{ADI} iterations per step does not increase by much beyond Newton step 10
due to the classical bound becoming active in formula~\eqref{eq:newton:hybrid} for the last three Newton iterations.
However,
the linear systems to be solved for the \ac{ADI} increments~\eqref{eq:lyapunov:increments} have more columns for the non-zero initial guess;
\qty{6}{\timesunit} as many (78 vs 13) at Newton step~7.
Despite this imbalance,
the new \ac{ADI} requires slightly fewer linear system solves overall (6457 vs 6708).
In total, we observe a \qty{6.0}{\second} reduction in $t_\mathrm{solve}$
at the cost of a \qty{1.4}{\second} increase in $t_\mathrm{compress}$.
In the future, we need to optimize our implementations computing the heuristic shifts as well as the low-rank column compression.

\begin{figure}
  \small
  \pgfplotsset{
    every axis/.append style={
      cycle list name=mymarks,
      width = .82\linewidth,
      height = .1\linewidth,
      scale only axis,
      yticklabel style = {overlay},
    },
  }
  \pgfplotstableread{plots/newton-adi/data/Rail5177_adi_kwargs=_maxiters=1000,shifts=Cyclic_Heuristic_10,_10,_10____newton_kwargs=_inexact=true,inexact_hybrid=true,linesearch=true,maxiters=20,reltol=1e-10__beta=1000.csv}\tablefigone

  \centering
  \begin{tikzpicture}
  \begin{axis}[%
    title = {Number of columns $z_\ell$ per Newton step.},
    xticklabel = \empty,
    ]

    \addplot table[x index = 0, y index = 7] {\tablefigone};
    \addplot table[x index = 0, y index = 8] {\tablefigone};
  \end{axis}
\end{tikzpicture}
  \begin{tikzpicture}
  \begin{axis}[%
    title = {Number of \ac{ADI} iterations until convergence per Newton step.},
    ymin = -5,
    ymax = 85,
    xticklabel = \empty,
    legend to name = {fig:are:macro:legend},
    ]

    \addplot table[x index = 0, y index = 1] {\tablefigone};
    \label{fig:are:macro:initzero}
    \addplot table[x index = 0, y index = 2] {\tablefigone};
    \label{fig:are:macro:initprev}
    \legend{old \ac{ADI}, new \ac{ADI}}
  \end{axis}
\end{tikzpicture}
  \begin{tikzpicture}
  \begin{axis}[
    title = {Size of the linear systems comprising the \ac{ADI} increment per Newton step.},
    ymin = -5,
    ymax = 85,
    xticklabel = \empty,
    ]

    \addplot table[x index = 0, y index = 3] {\tablefigone};
    \addplot table[x index = 0, y index = 4] {\tablefigone};
  \end{axis}
\end{tikzpicture}
  \begin{tikzpicture}
  \begin{axis}[
    title = {Total number of linear solves up to the Newton step.},
    xlabel = {Newton step},
    ]

    \addplot table[x index = 0, y index = 5] {\tablefigone};
    \addplot table[x index = 0, y index = 6] {\tablefigone};
  \end{axis}
\end{tikzpicture}
  \ref*{fig:are:macro:legend}
  \caption{%
    Newton method (hybrid w/ line search) applied to solve \ac{ARE}~\eqref{eq:are} arising from \autoref{ex:are}.
    \ac{ADI} shifts: $\heuristic(10, 10, 10)$.
    The \ac{ADI} is started from a zero value (\ref*{fig:are:macro:initzero})
    or with the previous Newton iterate (\ref*{fig:are:macro:initprev}).
  }%
  \label{fig:are:macro}
\end{figure}

If we plot the \Riccati{} and \Lyapunov{} residuals for all inner \ac{ADI} iterations directly one after the other,
we obtain \autoref{fig:are:micro}.
When starting the next Newton step's \ac{ADI} with the current Newton iterate,
$X_{\ell+1,0} = X_\ell$,
the corresponding \Riccati{} residual obviously does not change.
Meanwhile, when starting with the zero matrix,
$X_{\ell+1,0} = 0$,
the normalized \Riccati{} residual will jump back to~$1$.
Both effects are clearly visible in \autoref{fig:are:micro};
for the curve of the old \ac{ADI} it marks the start of every Newton iteration.
The curve of the new \ac{ADI} jumps back to~$1$ only due to the single line search applied after the first Newton step.
Furthermore, we observe a stagnation of the naive \Riccati{} residual~\eqref{eq:are:naive residual}
during the \ac{ADI} towards the end of every Newton step irrespective of its initial value (solid lines),
even though the \Lyapunov{} residual keeps getting smaller (dashed lines).
\citeauthorBenHSetal{} describe a more efficient representation of the \Riccati{} residual~$\Ricc(\cdot)$
in relation to the \Lyapunov{} residual~$\Lyap_\ell(\cdot)$ and the change in
the thin rectangular matrix $B^\TT X_\ell E \in\R^{m\times n}$~\cite[Equation~5.2b]{BenHSetal16}.
This representation allows to monitor the outer Riccati residual in the inner
\ac{ADI} iteration, which we plan to add to our implementation in the
future. Then the ADI can stop early whenever a sufficient decrease condition for
the Newton step is fulfilled.

We recommend to investigate
the cause of the aforementioned stagnation of the Riccati residual.
Furthermore, a hybrid \ac{ADI} approach should be studied,
that switches from the zero initial value to the non-zero one,
once the hybrid Newton criterion~\eqref{eq:newton:hybrid} resolves
to the classical condition~\eqref{eq:newton:classical}.

\begin{figure}
  \small
  \centering
  \begin{tikzpicture}
  \pgfplotstableread{plots/newton-adi/data/normalized_residual/Rail5177_adi_initprev=true_adi_kwargs=_maxiters=1000,shifts=Cyclic_Heuristic_10,_10,_10____newton_kwargs=_inexact=true,inexact_hybrid=true,linesearch=true,maxiters=20,reltol=1e-10__beta=1000.csv}\tablefigtwoprev
  \pgfplotstableread{plots/newton-adi/data/normalized_residual/Rail5177_adi_initprev=false_adi_kwargs=_maxiters=1000,shifts=Cyclic_Heuristic_10,_10,_10____newton_kwargs=_inexact=true,inexact_hybrid=true,linesearch=true,maxiters=20,reltol=1e-10__beta=1000.csv}\tablefigtwozero
  \begin{axis}[
    width = .82\linewidth,
    height = 0.41\linewidth,
    scale only axis,
    mark repeat = 50,
    xlabel = {Overall \ac{ADI} step},
    yticklabel style = {overlay},
    ymode = log,
    ytickten = {-12, -8, ..., 0},
    cycle list name = fourlines,
    legend transposed = true,
    legend cell align = left,
    legend style={font=\scriptsize,at={(0.98,0.02)},anchor=south east}
    ]

    \addplot table[x index = 0, y index = 2] {\tablefigtwozero};
    \addlegendentry{Riccati residial; old \ac{ADI}};
    \addplot table[x index = 0, y index = 1] {\tablefigtwozero};
    \addlegendentry{\Lyapunov{} residial; old \ac{ADI}};

    \addplot table[x index = 0, y index = 2] {\tablefigtwoprev};
    \addlegendentry{Riccati residial; new \ac{ADI}};
    \addplot table[x index = 0, y index = 1] {\tablefigtwoprev};
    \addlegendentry{\Lyapunov{} residial; new \ac{ADI}};
  \end{axis}
\end{tikzpicture}
  \caption{%
    Normalized \Riccati{} residuals $\norm{\Ricc(X_{\ell+1,k})} / \norm{\Ricc(0)}$
    following the naive formula~\eqref{eq:are:naive residual}, and
    normalized \Lyapunov{} residuals $\norm{\Lyap_\ell(X_{\ell+1,k})} / \norm{\Lyap_\ell(0)}$
    following \autoref{alg:lyapunov},
    over the course of all \ac{ADI} iterations~$k$ during all Newton steps~$\ell$.
    Newton method: hybrid w/ line search.
    \ac{ADI} shifts: $\heuristic(10, 10, 10)$.
    The \ac{ADI} is started from a zero value (old)
    or with the previous Newton iterate (new).
  }%
  \label{fig:are:micro}
\end{figure}

\subsection{Differential \Riccati{} Equation}

We apply a first-order Rosenbrock scheme, the implicit Euler method,
to the \acf{DRE}
\begin{equation}
  E^\TT \dot X E = \Ricc(X) := C^\TT C + A^\TT X E + E^\TT X A - E^\TT X BB^\TT X E
  ,
  \qquad
  E^\TT X(t_0) E = C^\TT C
  ,
  \label{eq:dre}
\end{equation}
where, again,
\begin{equation}
  E, A \in\R^{n\times n},
  \quad
  B \in\R^{n\times m},
  \quad
  C \in\R^{q\times n},
  \label{eq:dre:matrices}
\end{equation}
and $m,q\ll n$.
As $C^\TT C$ and $BB^\TT$ are positive semi-definite,
equation~\eqref{eq:dre} has a unique solution;
see, \eg~\cite[Theorem~4.1.6]{AboFIetal03}.
Motivated by Stillfjord~\cite{Sti18} and following \citeauthorLanMS~\cite{LanMS15},
we factorize the solution according to
\begin{math}
  X_\ell = Z_\ell Y_\ell Z_\ell^\TT
  \approx
  X(t_0+\ell\tau)
\end{math}
with
\begin{equation}
  Z_\ell \in\R^{n\times z_\ell},
  \quad
  Y_\ell \in\R^{z_\ell\times z_\ell},
\end{equation}
where $z_\ell \ll n$.
Thus, every Rosenbrock step reads
\begin{equation}
  \Lyap_\ell(X_{\ell+1}) :=
  G_\ell S_\ell G_\ell^\TT +
  A_\ell^\TT X_{\ell+1} E + E^\TT X_{\ell+1} A_\ell
  = 0
\label{eq:rosenbrock:step}
\end{equation}
where
\begin{equation}
\begin{alignedat}{2}
  A_\ell &= A - \tfrac{1}{2\tau} E - BB^\TT X_\ell E
         &&\in\R^{n\times n}
  \\
  G_\ell &= \begin{bmatrix}
    C^\TT & E^\TT Z_\ell
  \end{bmatrix}
         &&\in\R^{n\times (q + z_\ell)}
  \\
  S_\ell &=
  \begin{bmatrix}
    I \\
    & Y_\ell Z_\ell^\TT BB^\TT Z_\ell Y_\ell + \tfrac{1}{\tau} Y_\ell
  \end{bmatrix}
         &&\in\R^{(q+z_\ell)\times(q+z_\ell)}
  .
\end{alignedat}
\label{eq:dre:adi:matrices}
\end{equation}
Note, again, that $A_\ell$ has the structure of a sparse matrix updated by a low-rank matrix product.
This time, however, the inner dimension of the right-hand side $G_\ell S_\ell G_\ell^\TT$,
and therefore the corresponding \Lyapunov{} residual~$\Lyap_\ell(X_{\ell+1,k})$
both depend on the rank of the current Rosenbrock iterate~$X_\ell$.
More precisely,
the rank is at most $(q+z_\ell)+2z_\ell$;
see formula~\eqref{eq:lyapunov:residual:structure}
\chadded{with $g=q+z_\ell$ and $z=z_\ell$}.
Again, we do not yet utilize the common $E^\TT Z_\ell$ in the outer low-rank factor,
which would allow for a factorization having inner dimension $q+2z_\ell < q+3z_\ell$ and, thus, reduced compression times.

\begin{example}%
\label{ex:dre}
We apply the procedure described above to the Steel Profile benchmark~\cite{morwiki_steel,BenS05b},
see \autoref{ex:are}.
The time span is $[t_0, t_f] = [0, 4500]$,
which we discretize in $\ell_{\max}=45$ or $450$ equidistant segments.
That is, $\tau=100$ or $10$.
The convergence criterion~\eqref{eq:adi:convergence} is based on $\reltol_\textup{\ac{ADI}} = 10^{-10}$;
\end{example}

\afterpage{%
\begin{landscape}
\begin{table}
  \centering
  \caption{%
    Total number of \ac{ADI} iterations~$k_\mathrm{total}$ performed,
    \chadded{the maximum number of columns~$\max_\ell z_\ell$ of the outer solution factor (maximum rank),}
    as well as several run-time metrics
    to solve the \ac{DRE}~\eqref{eq:dre} arising from \autoref{ex:dre}.
    All timings are in seconds.
    The \ac{ADI} is started from a zero value (old)
    or with the previous Rosenbrock iterate (new).
    A speedup \num{>1} indicates that our new \ac{ADI} is faster.
  }%
  \label{tab:dre}
  \resizebox{\linewidth}{!}{%
  \begin{tabular}{%
    S[table-format=3]
    l
    S[table-format=5]S[table-format=4]
    S[table-format=3]S[table-format=3]
    S[table-format=5.1]S[table-format=3.1]
    S[table-format=3.1]S[table-format=3.1]
    S[table-format=4.1]S[table-format=3.1]
    S[table-format=5.1]S[table-format=4.1]
    S[round-precision=2,table-format=2.2]
  }
    \toprule
    {\multirow{2}{\widthof{\#Rosenbrock}}[-2pt]{\centering \#Rosenbrock steps}} &
    {\multirow{2}{*}[-2pt]{\ac{ADI} shifts}} &
    \multicolumn{2}{c}{\#ADI steps} &
    \multicolumn{2}{c}{Rank} &
    \multicolumn{2}{c}{$t_\mathrm{compress}$} &
    \multicolumn{2}{c}{$t_\mathrm{shifts}$} &
    \multicolumn{2}{c}{$t_\mathrm{solve}$} &
    \multicolumn{2}{c}{$t_\mathrm{total}$} &
    {\multirow{2}{*}[-2pt]{Speedup}}
    \\
    \cmidrule(lr){3-4}
    \cmidrule(lr){5-6}
    \cmidrule(lr){7-8}
    \cmidrule(lr){9-10}
    \cmidrule(lr){11-12}
    \cmidrule(lr){13-14}
&
&
    {old} & {new} &
    {old} & {new} &
    {old} & {new} &
    {old} & {new} &
    {old} & {new} &
    {old} & {new} &
    \\\midrule
    45 & $\heuristic(10, 10, 10)$ & 1395 & 671 & 145 & 152 & 211.8 & 41.2 & 7.6 & 7.6 & 53.5 & 18.8 & 314.7 & 89.5 & 3.52 \\
45 & $\heuristic(20, 30, 30)$ & 1170 & 683 & 144 & 152 & 177.8 & 43.4 & 24.3 & 24.1 & 44.7 & 19.0 & 282.5 & 109.9 & 2.57 \\
45 & $\xdecrproj{2}$ & 5685 & 550 & 144 & 152 & 971.1 & 72.7 & 3.7 & 2.1 & 222.5 & 19.3 & 1357.7 & 119.6 & 11.35 \\
45 & $\xheurproj{2}$ & 1266 & 821 & 145 & 152 & 201.0 & 49.3 & 3.9 & 1.6 & 49.9 & 22.2 & 293.7 & 97.3 & 3.02 \\
45 & $\xincrproj{2}$ & 4247 & 2172 & 142 & 152 & 670.3 & 181.8 & 3.7 & 1.7 & 164.4 & 64.1 & 958.7 & 296.0 & 3.24 \\
450 & $\heuristic(10, 10, 10)$ & 9450 & 3473 & 144 & 155 & 1572.3 & 127.9 & 172.5 & 133.8 & 562.4 & 93.3 & 2742.8 & 536.7 & 5.11 \\
450 & $\heuristic(20, 30, 30)$ & 9000 & 3392 & 144 & 155 & 1626.2 & 139.6 & 243.6 & 285.9 & 487.2 & 93.8 & 2849.3 & 706.9 & 4.03 \\
450 & $\xdecrproj{2}$ & 54657 & 1306 & 141 & 154 & 10727.1 & 73.6 & 38.2 & 8.2 & 3525.8 & 54.0 & 16405.3 & 313.2 & 52.38 \\
450 & $\xheurproj{2}$ & 8247 & 3739 & 144 & 152 & 1511.3 & 129.8 & 38.6 & 10.1 & 538.5 & 119.0 & 2484.7 & 469.9 & 5.29 \\
450 & $\xincrproj{2}$ & 25487 & 9693 & 144 & 151 & 4677.9 & 396.3 & 37.3 & 12.4 & 1646.7 & 322.1 & 7401.2 & 1037.6 & 7.13 \\

    \bottomrule
  \end{tabular}%
  }
\end{table}
\end{landscape}
}

\autoref{tab:dre} gives an overview on the number of \ac{ADI} iterations required
to solve the \ac{DRE}~\eqref{eq:dre} arising from \autoref{ex:dre}.
The timings are as described for \autoref{tab:are}.
Refer to the beginning of this section for an explanation on the shift strategies chosen.
Changing the initial \ac{ADI} guess from the zero matrix to the previous Rosenbrock iterate~$X_{\ell-1}$
always reduced the overall number of \ac{ADI} iterations.
In contrast to the previous section,
the run-time is always dominated by $t_\mathrm{compress}$.
This motivates, again, to improve the implementation of the low-rank column compression.

For this application, both \acp{ADI} are susceptible to the order of the shifts.
Focusing on the projection shifts for \num{450} Rosenbrock steps,
the number of \ac{ADI} iterations varies by more than \qty{6}{\timesunit} by merely changing the order of the shifts (old: \num{54657} vs \num{8247}, new: \num{9693} vs \num{1306}).
In terms of the number of iterations,
the old \ac{ADI} performs best for $\xheurproj{2}$ shifts and
appears to degenerate for $\xdecrproj{2}$ shifts.
Meanwhile, the new \ac{ADI} performs best for this order;
taking less than 3 steps on average;
and mediocre for $\xheurproj{2}$ shifts.

\autoref{fig:dre} shows some more detail for one of the configurations of the table;
the general shape is the same for all configurations.
The \ac{DRE}~\eqref{eq:dre} is very stiff during early Rosenbrock steps~$\ell\in\N$.
Once the integrator reaches a more transient regime for large-enough~$\ell$,
using the non-zero initial \ac{ADI} value reduces the number of \ac{ADI} steps drastically;
as expected.
Surprisingly, however,
the number of columns comprising the linear systems to compute
the \ac{ADI} increments~\eqref{eq:lyapunov:increments}
decreases as well.
Apparently, the previous Rosenbrock iterate~$X_\ell$ contains enough information
to cause the \Lyapunov{} residual~$\Lyap_\ell(X_\ell)$ to vanish in many directions of the corresponding subspace.
Overall, the new \ac{ADI} has to solve fewer \emph{and} cheaper linear systems.
Consequently,
the finer the temporal resolution,
the larger the expected speedup of the new \ac{ADI}.
\unskip\footnote{%
  Adaptive time stepping may reduce this advantage of our new \ac{ADI},
  as regions allowing for large time steps are the ones that would benefit most from our initial values.
  Therefore, as a stopgap specifically for \autoref{ex:dre},
  we select temporal resolutions that are rather coarse from an engineer's perspective.
}
This notion is confirmed by \autoref{tab:dre}.
When increasing the number of Rosenbrock steps by \qty{10}{\timesunit},
the total number of iterations and run-time of the new \ac{ADI}
only increase by roughly \qtyrange{2.5}{5}{\timesunit} and \qtyrange{2}{6}{\timesunit},
respectively, depending on the shifts.
Meanwhile,
the total number of iterations and run-time of the old \ac{ADI}
increase by roughly \qtyrange{6}{10}{\timesunit} and \qtyrange{7.5}{12}{\timesunit},
respectively, depending on the shifts.
Comparing the best shifts for each of the \acp{ADI} applied over \num{450} Rosenbrock steps,
we observe a \qty{6}{\timesunit} reduction in \ac{ADI} steps (\num{8247} vs \num{1306})
and an \qty{8}{\timesunit} reduction in run-time (\qty{2484.7}{\second} vs \qty{313.2}{\second})
in favor of our new \ac{ADI}.

\begin{figure}
  \small
  \pgfplotsset{
    every axis/.append style={
      cycle list name=mylines,
      width = .82\linewidth,
      height = .1\linewidth,
      scale only axis,
      xtick = {1, 50, 100, ..., 450},
      mark repeat = 50,
      yticklabel style = {overlay},
    },
  }
  \pgfplotstableread{plots/rosenbrock/data/Rail5177_adi_kwargs=_maxiters=200,reltol=1e-10,shifts=Cyclic_Heuristic_20,_30,_30____nsteps=450_tspan=_4500.0,_0.0_.csv}\tablefigthree
  \centering
  \begin{tikzpicture}
  \begin{axis}[
    title = {Number of columns $z_\ell$ per Rosenbrock step.},
    xticklabel = \empty,
    ]

    \addplot table[x index = 0, y index = 7] {\tablefigthree};
    \addplot table[x index = 0, y index = 8] {\tablefigthree};
  \end{axis}
\end{tikzpicture}
  \begin{tikzpicture}
  \begin{axis}[%
    title = {Number of \ac{ADI} iterations until convergence per Rosenbrock step.},
    ymin = -2,
    ymax = 22,
    xticklabel = \empty,
    legend to name = {fig:dre:legend},
    ]

    \addplot table[x index = 0, y index = 1] {\tablefigthree};
    \label{fig:dre:initzero}
    \addplot table[x index = 0, y index = 2] {\tablefigthree};
    \label{fig:dre:initprev}
    \legend{old \ac{ADI}, new \ac{ADI}}
  \end{axis}
\end{tikzpicture}
  \begin{tikzpicture}
  \begin{axis}[
    title = {Size of the linear systems comprising the \ac{ADI} increment per Rosenbrock step.},
    ymin = -15,
    ymax = 165,
    xticklabel = \empty,
    ]

    \addplot table[x index = 0, y index = 3] {\tablefigthree};
    \addplot table[x index = 0, y index = 4] {\tablefigthree};
  \end{axis}
\end{tikzpicture}
  \begin{tikzpicture}
  \begin{axis}[
    title = {Total number of linear solves up to the Rosenbrock step.},
    scaled y ticks = false,
    xlabel = {Rosenbrock step},
    ]

    \addplot table[x index = 0, y index = 5] {\tablefigthree};
    \addplot table[x index = 0, y index = 6] {\tablefigthree};
  \end{axis}
\end{tikzpicture}
  \ref*{fig:dre:legend}
  \caption{%
    Rosenbrock method applied to solve \ac{DRE}~\eqref{eq:dre} arising from \autoref{ex:dre}.
    \ac{ADI} shifts: $\heuristic(20, 30, 30)$.
    Rosenbrock step size $\tau=10$.
    The \ac{ADI} is started from a zero value (\ref*{fig:dre:initzero})
    or with the previous Rosenbrock iterate (\ref*{fig:dre:initprev}).
  }%
  \label{fig:dre}
\end{figure}

Note that, although prescribed,
the order of the heuristic shifts does not matter that much.
The heuristic strategies compute only few (here: 10 or 20) shifts at once.
Once all of them have been used,
the iterates and residuals coincide;
see \autoref{thm:permutation}.
The projection strategies, however,
compute twice the inner dimension of the \Lyapunov{} residual many at once (here: approx.~\num{80} to \num{250}).
For the vast majority of Rosenbrock steps,
the \ac{ADI} does converge after only a subset of these shifts has been used;
that is, \autoref{thm:permutation} does not apply.
Further research is needed to understand the effect of the shift order in such a scenario,
and how to select the optimal shift order (and subset of shifts).

\section{Conclusion}%
\label{sec:conclusion}

We introduced the notion of fully commuting splitting schemes to solve arbitrary linear systems,
and derived the \ac{ADI} method in that context.
This allowed us to extend the low-rank \Lyapunov{} \ac{ADI} for complex data to non-zero initial values.
Furthermore, we generalized the permutation invariance of \ac{ADI} iterates
to arbitrary fully commuting splitting schemes (\autoref{thm:permutation}),
as well as the existence of a real-valued \ac{ADI} double-step for complex-conjugated shifts
to arbitrary linear systems (\autoref{thm:adi:2step}).

We applied the extended low-rank \Lyapunov{} \ac{ADI} to a Newton and a Rosenbrock method
to solve an algebraic and a differential \Riccati{} equation, respectively.
For the Newton method we observed a \qty{4}{\timesunit} lower total number of \ac{ADI} steps for heuristic Penzl shifts,
but more expensive linear systems at every \ac{ADI} step.
Overall, for this application, our implementation only showed a modest \qty{17}{\%} run-time improvement over the old \ac{ADI}.
\chdeleted{%
Therefore, at the moment we doubt that our modifications to the Newton-ADI for \acp{ARE}
suffice to make the method competitive with the \ac{RADI}~\cite{BenBKetal18}.
}
For the Rosenbrock method, however, we observed a \qtyrange{2}{6}{\timesunit} lower total number of \ac{ADI} steps,
depending on the shifts,
whose linear systems have fewer columns,
resulting in an \qty{8}{\timesunit} speed-up in our implementation.

\section*{Code and Data Availability}

The algorithms have been implemented using Julia~\cite{BezEKetal17}
and are available at:
\begin{center}
  DOI \href{https://doi.org/10.5281/zenodo.10650859}{10.5281/zenodo.10650859}
\end{center}
The datasets analyzed in this paper are available at:
\begin{center}
  \begin{tabular}{rl}
    Newton: & DOI \href{https://doi.org/10.5281/zenodo.10650872}{10.5281/zenodo.10650872} \\
    Rosenbrock: & DOI \href{https://doi.org/10.5281/zenodo.10651124}{10.5281/zenodo.10651124}
  \end{tabular}
\end{center}

\section*{Acknowledgments}

We would like to thank Daniel Szyld~\orcidlink{0000-0001-8010-0391}
for pointing out the name \emph{nonstationary} splitting scheme.
We further thank Fan Wang~\orcidlink{0009-0000-8943-6908}
and Martin K\"{o}hler~\orcidlink{0000-0003-2338-9904} for their review of our codes
and independent verification of the reproducibility of the numerical experiments.

\appendix
\section{Order of \acs{ADI} shifts}%
\label{sec:app:shift order}

As we have observed in \autoref{sec:applications},
the low-rank \Lyapunov{} \ac{ADI} can be quite sensitive to the order of its shifts.
In this section,
we give some intuition on why the projection shifts ordered by increasing real part, $\xincrproj{2}$,
sometimes lead to large numbers of \ac{ADI} iterations.

Using the notation of \autoref{sec:splitting schemes},
the iteration map of the symmetry-preserving low-rank \Lyapunov{} \ac{ADI},
as derived in \autoref{sec:lyapunov},
is given by
\begin{equation}
  G_k(U) :=
  (M_k^{-1} N_k)(U) =
  \Cayley(A, \alpha_{k}) \, U \, \Cayley(A, \alpha_{k})^\HT
  ,
\end{equation}
using the operator split~\eqref{eq:lyapunov:opsplit}, $\beta_k := \conj{\alpha_{k}}$, and the Cayley transformation
\begin{equation}
  \Cayley(A, \alpha) := (A + \alpha I)^{-1} (A - \conj{\alpha} I)
  .
\end{equation}
If applied to a symmetric low-rank factorization $ZYZ^\HT$,
the iteration map effectively only operates on the outer factors~$Z$.
Recall that a nonstationary splitting scheme converges iff $\rho(G_k\cdots G_0) \to 0$ as $k\to\infty$.
A sufficient condition is thus
\begin{equation}
  \rho\big(
    \Cayley(A, \alpha_{k})
    \cdots
    \Cayley(A, \alpha_{0})
  \big)
  \to 0
\end{equation}
as $k\to\infty$.
As the spectral radius is sub-multiplicative, the upper bound
\begin{equation}%
\label{eq:app:upper bound}
  \rho\big(
    \Cayley(A, \alpha_{k})
    \cdots
    \Cayley(A, \alpha_{0})
  \big)
  \leq
  \rho\big( \Cayley(A, \alpha_{k}) \big)
  \cdots
  \rho\big( \Cayley(A, \alpha_{0}) \big)
  =: \hat \rho_k
\end{equation}
holds. It can be computed as follows.
By, \eg~\cite[Proposition~2.16]{Kue16},
the spectral radius of such a Cayley transformation is given as
\begin{equation}
  \rho(\mathcal C(A, \alpha)) =
  \max\left\{
    \frac{\abs{\lambda - \conj{\alpha}}}{\abs{\lambda + \alpha}} :
    \lambda\in\Lambda(A)
  \right\}
  .
\end{equation}

If we chose the parameters $\{\alpha_0, \ldots, \alpha_k\}$
to be (a subset of) the spectrum $\Lambda(A)$,
it holds $\rho(\Cayley(A, \alpha_k)) \leq 1$ for any~$k$.
Furthermore,
every eigenvector~$v$ of~$A$ to the eigenvalue~$\lambda$
lives in the null space of~$\Cayley(A, \lambda)$.
Consequently, if the whole $\Lambda(A)$ is chosen,
the norm of the combined step
\begin{math}
  \rho\big(
    \Cayley(A, \alpha_{k})
    \cdots
    \Cayley(A, \alpha_0)
  \big)
\end{math}
will be zero.
On the other hand,
any individual $\rho(\Cayley(A, \alpha_k)) > 0$.
That is, $\hat\rho_k$ is not a sharp upper bound.

\autoref{fig:app:combined} shows the aforementioned upper bound~$\hat\rho_0 = \rho(\Cayley(A,\alpha_0))$ and~$\hat\rho_k$
for a smaller configuration of the Steel Profile~\cite{morwiki_steel,BenS05b} benchmark having matrix dimension~$n=371$.
The corresponding (generalized) spectrum is contained in the negative half plane,
and the whole set is chosen as the parameter set $\{\alpha_0, \ldots, \alpha_{n-1}\}$.
These parameters may be ordered
by increasing real part (similar to $\xincrproj{2}$),
by decreasing real part (similar to $\xdecrproj{2}$),
or by Penzl's heuristic (similar to $\xheurproj{2}$);
see \autopageref{item:shifts:proj:heur}.
Note that Penzl's heuristic~\cite{Pen00b}
chooses the first shift~$\alpha_0$ to be the one minimizing~$\rho(\Cayley(A,\cdot)) = \hat\rho_0$,
and continues greedily to select~$\alpha_k$ by minimizing
\begin{equation}%
\label{eq:app:penzl}
  \prod_{i=0}^{k-1}
  \frac{\abs{\alpha_k - \alpha_i}}{\abs{\alpha_k + \alpha_i}}
\end{equation}
for given $\{\alpha_0, \ldots, \alpha_{k-1}\}$.
\unskip\footnote{%
  The heuristic includes $\alpha_{k+1} := \conj{\alpha_k}$ if $\alpha_k\in\C\setminus\R$.
  For the present example, however, this does not occur.
}
Quantity~\eqref{eq:app:penzl} is related to~$\hat\rho_k$, but not the same.
\unskip\footnote{%
  Note the lack of conjugation in the numerators.
  Furthermore, each factor containing~$\alpha_i$
  is in general only a lower bound to~$\rho(\Cayley(A, \alpha_i))$,
  which renders quantity~\eqref{eq:app:penzl} smaller than $\hat\rho_{k-1}$.
}
Consequently, the heuristic is the only one for which
the corresponding~$\hat\rho_0$ is visibly smaller than~$1$.

\begin{figure}
  \centering
  \pgfplotstableread{plots/appendix/single_step_norm.csv}\tablefigappsingle
\pgfplotstablecreatecol[%
  create col/expr={-log10(-\thisrow{eigenvalue})}, 
]{neigenvalue}\tablefigappsingle
\pgfplotstableread{plots/appendix/combined_step_norm.csv}\tablefigappcombined
\begin{tikzpicture}[baseline]
  \begin{axis}[
    title = {Spectral radius $\rho(\Cayley(A, \lambda))$},
    xlabel = {\strut Eigenvalue $\lambda\in\Lambda(A)$},
    xtick = {-1, 1, ..., 5},
    xticklabel = {\strut$-10^{\pgfmathprintnumber{-\tick}}$},
    y tick label style = {/pgf/number format/precision = 4},
    cycle list name = mymarks,
    width = 0.4\linewidth,
    scale only axis,
    ]

    \addplot+[mark = x] table[x = neigenvalue, y index = 1] {\tablefigappsingle};
  \end{axis}
\end{tikzpicture}
\hfill
\begin{tikzpicture}[baseline]
  \begin{axis}[
    title = {Upper bound $\hat\rho_k$},
    xlabel = {\strut\ac{ADI} step~$k$},
    typeset ticklabels with strut,
    no markers,
    cycle list = {
      {mycolor1},
      {mycolor2},
      {mycolor3},
    },
    legend cell align = left,
    legend columns = 1,
    legend style = {font = \scriptsize, opacity = .8},
    width = 0.4\linewidth,
    scale only axis,
    ]

    \addplot table[x expr = \coordindex, y index = 1] {\tablefigappcombined};
    \addlegendentry{increasing real part}
    \addplot table[x expr = \coordindex, y index = 2] {\tablefigappcombined};
    \addlegendentry{Penzl's heuristic}
    \addplot table[x expr = \coordindex, y index = 0] {\tablefigappcombined};
    \addlegendentry{decreasing real part}
  \end{axis}
\end{tikzpicture}
  \caption{%
    Spectral radius of Cayley transformations associated to spectrum of~$A$, and
    upper bound $\hat\rho_k$ on the norm of parts of the iteration map
    over the course of multiple \ac{ADI} iterations~$k$,
    for different permutations of the spectrum~$\Lambda(A)$.
  }%
  \label{fig:app:combined}
\end{figure}

Obviously, all upper bounds~$\hat\rho_k$ coincide for $k=n-1$,
as the underlying product of spectral radii~\eqref{eq:app:upper bound}
is formed over the whole set of parameters.
The decreasing order clearly yields a better a priori upper bound than the increasing order.
Interestingly, the heuristic order starts out smallest,
remains in between the other two orders for most iterations,
but comes out largest towards the last iterations.
Thus, for most cases,
we expect the $\xdecrproj{2}$ shifts to require the least number of \ac{ADI} steps until convergence,
followed by $\xheurproj{2}$ and $\xincrproj{2}$, in that order.
Unfortunately, however, this intuition can only be confirmed for the new \ac{ADI} in Tables~\ref{tab:are} and~\ref{tab:dre}.

\addcontentsline{toc}{section}{References}
\bibliographystyle{siam}
\bibliography{manuscript-etna.bib}

\end{document}